\documentclass[11pt,a4paper,reqno]{amsart}
\usepackage{amsaddr}
\usepackage{amsmath,amsfonts,amssymb,amsthm,mathrsfs}
\usepackage[utf8]{inputenc}
\usepackage[T1]{fontenc}
\usepackage{lmodern}
\usepackage{latexsym}
\usepackage{cancel}
\usepackage{microtype}
\usepackage{caption}
\usepackage{MnSymbol}
\usepackage{xcolor}
\usepackage{enumerate}
\usepackage[normalem]{ulem}
\usepackage{bbm}
\usepackage{geometry}
\usepackage{marginnote}

\usepackage[colorlinks=true,pdfpagemode=none,urlcolor=blue,
linkcolor=blue,citecolor=blue]{hyperref}

\usepackage{graphics,tikz} 
\usepackage{caption}

\newcommand{\FF}{{\mathcal  F}}

\newcommand{\TT}{{\mathcal  T}}

\def\XXint#1#2#3{{\setbox0=\hbox{$#1{#2#3}{\int}$ }
\vcenter{\hbox{$#2#3$ }}\kern-.6\wd0}}

\newtheorem{theorem}{\bf Theorem}[section]
\newtheorem{proposition}[theorem]{\bf Proposition}
\newtheorem{lemma}[theorem]{\bf Lemma}
\newtheorem{corollary}[theorem]{\bf Corollary}

\theoremstyle{definition}
\newtheorem{definition}[theorem]{Definition}

\newtheorem{remark}[theorem]{Remark}

\numberwithin{equation}{section}

\begin{document}

\title[The existence result for  BSDEs with two optional Doob's class barriers]{Nonlinear BSDEs with two optional Doob's class barriers satisfying weak Mokobodzki's condition and  extended Dynkin games}

\maketitle
\begin{center}
 \normalsize
  TOMASZ KLIMSIAK\footnote{e-mail: {\tt tomas@mat.umk.pl}}\textsuperscript{1,2} \,\,\,
  MAURYCY RZYMOWSKI\footnote{e-mail: {\tt maurycyrzymowski@mat.umk.pl}}\textsuperscript{2}   \par \bigskip
  \textsuperscript{\tiny 1} {\tiny Institute of Mathematics, Polish Academy of Sciences,\\
 \'{S}niadeckich 8,   00-656 Warsaw, Poland} \par \medskip
 
  \textsuperscript{\tiny 2} {\tiny Faculty of
Mathematics and Computer Science, Nicolaus Copernicus University,\\
Chopina 12/18, 87-100 Toru\'n, Poland }\par
\end{center}

\begin{abstract}
We study reflected backward stochastic differential equations (RBSDEs) on the probability space equipped with a Brownian motion. 
The main novelty of the paper lies in the fact that we consider the following weak assumptions on the data: barriers are  optional 
of class (D) satisfying weak Mokobodzki’s condition,   generator is   continuous and non-increasing with respect to the value-variable 
(no restrictions on the growth) and Lipschitz continuous with respect to the control-variable, and the terminal condition and the generator 
at zero are supposed to be merely integrable. We prove that under these conditions on the data there exists a  solution to corresponding  RBSDE. 
In the second part of the paper, we  apply the theory of RBSDEs to solve basic problems  in  
Dynkin games driven by  nonlinear expectation based on the generator mentioned above. We prove that the 
main component of a solution to RBSDE represents  the value process in corresponding extended nonlinear Dynkin game.   
Moreover, we provide sufficient conditions on the barriers guaranteeing the existence of the value for nonlinear Dynkin games and the existence of a saddle point.
\end{abstract}
\maketitle


\section{Introduction}
\label{wstep}
Let $B$ be a standard $d$-dimensional Brownian motion on a given probability space $(\Omega,\FF,P)$,
$T$ be a strictly positive real number (horizon time) and let $\mathbb{F}:=(\FF_t)_{0\le t\le T}$
be the standard augmentation of the  filtration generated by $B$. In the present paper, we study Reflected Backward Stochastic Differential Equations (RBSDEs for short) of the following form
\begin{equation}\label{2marca1}
\begin{cases}
Y_t=\xi+\int^T_t f(r,Y_r,Z_r)\,dr+R_T-R_t-\int^T_t Z_r\,dB_r,\quad t\in[0,T],\\
L_t\le Y_t\le U_t,\quad t\in[0,T],
\end{cases}
\end{equation}
where $\xi$ (terminal value) is an $\mathcal{F}_T$-measurable random variable, the mapping 
$f:\Omega\times[0,T]\times\mathbb{R}\times\mathbb{R}^d\to\mathbb{R}$ (generator) 
is an $\mathbb{F}$-progressively measurable process with respect to the first two variables and 
$L,U$ (barries) are $\mathbb{F}$-optional processes of class (D). We look for a triple $(Y,Z,R)$ of 
$\mathbb{F}$-progressively measurable processes, with $R$ of finite variation and $R_0=0$, that satisfies \eqref{2marca1}. 
Given a solution  $(Y,Z,R)$ to \eqref{2marca1}, we call the process $Y$   the {\em main part} of the solution.
The role of $R$ is to keep $Y$ between barriers $L,U$,
and the role of $Z$ is to keep $Y$ adapted to $\mathbb F$.  In order to get the uniqueness for problem \eqref{2marca1} one requires  from
$R$ to satisfy the so called {\em minimality condition} which stands that
\begin{equation}
\label{eq.2903}
\begin{split}
&\int^T_0(Y_{r-}-\limsup_{s\uparrow r}L_s)\,dR^{*,+}_r+\int^T_0(\liminf_{s\uparrow r}U_s-Y_{r-})\,dR^{*,-}_r=0\\&
\sum_{0\le r<T}(Y_r-L_r)\max\{R_{r+}-R_r,0\}+\sum_{0\le r<T}(U_r-Y_r)\max\{R_r-R_{r+},0\}=0,
\end{split}
\end{equation}
where $R^*$ is the c\`adl\`ag part of $R$ and $R^{*,+}, R^{*,-}$ its Jordan decomposition.

\medskip
{\bf  Formulation of the problems.} In the  paper, we merely assume that
\begin{enumerate}
\item[(A1)] $\mathbb E|\xi|+\mathbb E\int_0^T|f(r,0,0)|\,dr<\infty$,
\item[(A2)] there is $\lambda\ge0$ such that
$|f(t,y,z)-f(t,y,z')|\le\lambda|z-z'|$ for $t\in[0,T]$, $y\in\mathbb{R}$, $z,z'\in\mathbb{R}^d$,
\item[(A3)]  $y\mapsto f(t,y,z)$ in non-increasing and continuous  for fixed  $t\in[0,T]$, $z\in\mathbb{R}^d$,
\item[(A4)] $\int^T_0 |f(r,y,0)|\,dr<\infty$ for every $y\in\mathbb{R}$,
\item[(Z)] there exist  $\gamma\ge 0$, $\kappa\in [0,1)$ and  a non-negative  $\mathbb F$-progressively measurable process $g$,
satisfying $\mathbb E\int_0^Tg_r\,dr<\infty$, such that
\begin{align*}
|f(t,y,z)-f(t,y,0)|\le\gamma(g_t+|y|+|z|)^\kappa,\quad t\in [0,T], \,y\in\mathbb{R},\,z\in\mathbb{R}^d.
\end{align*}
\end{enumerate}
By \cite{bdh}, under (A1)--(A4),(Z), there exists a solution $(Y,Z)$  to  \eqref{2marca1} without barriers (BSDE), i.e. 
\begin{equation}
\label{2marca1pr}
Y_t=\xi+\int^T_t f(r,Y_r,Z_r)\,dr-\int^T_t Z_r\,dB_r,\quad t\in[0,T].
\end{equation}
Moreover, it is unique provided $Y$ is of class (D) and $Z\in \mathcal H^s_{\mathbb F}(0,T)$ - the class of $\mathbb F$-progressively measurable 
processes satisfying $\mathbb E[\int_0^T|Z_r|^2\,dr]^{s/2}<\infty$ - 
for some $s>\kappa$. In the present  paper, we focus on the existence and uniqueness problem for \eqref{2marca1}--\eqref{eq.2903}
under conditions (A1)--(A4),(Z).  We  shall also study the representation of the process $Y$ as the  value process in nonlinear Dynkin games.

\medskip
{\bf The existence problem.} First, observe that by the very definition of a solution to \eqref{2marca1} its {\em main part}
is a semimartingale. Consequently, we deduce at once, that the existence of  a semimartingale between the barriers $L, U$
is a necessary  condition for the existence of a solution to \eqref{2marca1} (intrinsic condition).  The said condition is known in the literature 
as {\em weak Mokobodzki's condition} (see \cite{HH1}):
\vspace*{0.15cm}
\begin{enumerate}
\item[(WM)] there exists a semimartingale $X$ such that $L_t\le X_t\le U_t,\, t\in [0,T]$.  
\end{enumerate}
\vspace*{0.15cm}
It is a natural question whether under (A1)--(A4), (Z) the  above condition is also sufficient for the existence of a solution to \eqref{2marca1}--\eqref{eq.2903}.
We give a positive answer to this question, and prove even more, that condition (Z) can be dropped. 

\vspace*{0.10cm}
\begin{center}
\begin{minipage}[c][1,15cm][t]{0,85\textwidth}
{\em 
\textbf{Theorem 1.} Assume that \textnormal{(A1)--(A4)} hold and weak Mokobodzki's condition \textnormal{(WM)} is in force.
Then there exist a solution $(Y,Z,R)$ to \eqref{2marca1}--\eqref{eq.2903}.
}
\end{minipage}
\end{center}
It appears, and it may seem  surprising at first,  that   the  above result does not hold for  BSDEs \eqref{2marca1pr} (see Remark \ref{remark.z}).
The explanation of this phenomenon  is that in the case of reflected BSDEs barriers keep the {\em main part} of a solution   
in the class (D).  At this point it is worth mentioning that the following condition (complete separation)
\begin{equation}
\label{cond1}
L,U\text{ are c\`adl\`ag,}\quad L_t< U_t,\,\, t\in [0,T],\quad L_{t-}<U_{t-},\,\, t\in (0,T]
\end{equation}
implies (WM) (see \cite[Lemma 3.1]{Topolewski}). 
We provide a generalization of this condition, by dropping c\`adl\`ag regularity assumption  on $L,U$, and we prove that
the following condition implies (WM):
\begin{equation}
\label{cond2}
L,U\text{ are left-limited,}\quad  L_{t-}< U_{t-},\quad \limsup_{s\downarrow t} L_s< \liminf_{s\downarrow t}U_s,\quad t\in [0,T].
\end{equation}

\medskip
{\bf The uniqueness problem.} An interesting issue is also the problem of the uniqueness
for solutions to \eqref{2marca1}--\eqref{eq.2903}. In the proof of the uniqueness for   BSDEs \eqref{2marca1pr} (see \cite{bdh} and Theorem \ref{12sierpnia1})
the crucial roles 
were played by condition (Z)
and the fact that for any solution $(Y,Z)$  to \eqref{2marca1pr} we have, under conditions (A1)--(A4),(Z), that  $Z\in \mathcal H_{\mathbb F}^s(0,T)$ for some $s>\kappa$
provided $Y$ is of class (D). For reflected BSDEs this property does not hold
even if $f\equiv 0$ (see \cite[Example 5.6]{kl2}). Nevertheless, we are able to prove the following result.
\begin{center}
\begin{minipage}[c][0,9cm][t]{0,85\textwidth}
{\em 
\textbf{Theorem 2.} Assume that (A1)--(A4),(Z) are in force.
Then there exists at most one  solution  to RBSDE \eqref{2marca1}--\eqref{eq.2903}. 
}
\end{minipage}
\end{center}

{\bf Solutions to RBSDEs  as  value processes in Dynkin games.} 
The above theorem is a consequence of a much deeper result, which is our third main result of the paper. 
In order to  formulate it, we use the notion of the nonlinear expectation introduced by Peng in \cite{Peng}. 
For given stopping times $\alpha\le \beta\le T$ consider mapping
\[
\mathbb{E}^{f}_{\alpha,\beta}:L^1(\mathcal{F}_{\beta})\to L^1(\mathcal{F}_{\alpha}),
\]
by letting $\mathbb{E}^{f}_{\alpha,\beta}\xi:=Y^\beta_\alpha$, where $(Y^\beta,Z^\beta)$ is a (unique) solution  
to \eqref{2marca1pr}, with $T$ replaced by $\beta$, such that $Y^\beta$ is of class (D). 
For given stopping times $\tau,\sigma\le T$ and sets $H\in\mathcal F_\tau, G\in\mathcal F_\sigma$, we let
\begin{equation*}
\begin{split}
J(\tau,H;\sigma, G)&:=(L_{\tau}\mathbf1_H+\limsup_{h\searrow 0}L_{\tau+h}\mathbf1_{H^c})\mathbf{1}_{\{\tau \le\sigma<T\}}+
(U_{\sigma}\mathbf1_G\\
&\quad+\liminf_{h\searrow 0}U_{\sigma+h}\mathbf1_{G^c})\mathbf{1}_{\{\sigma<\tau\}}+\xi\mathbf{1}_{\{\tau=\sigma=T\}},
\end{split}
\end{equation*}
with the convention that $L_t=L_{t\wedge T}, U_t:=U_{t\wedge T},\, t\ge 0$.
We prove the following representation theorem.
\vspace*{0.10cm}
\begin{center}
\begin{minipage}[c][2,55cm][t]{0,85\textwidth}
{\em
\textbf{Theorem 3.} Assume that \textnormal{(A1)--(A4), (Z)} are in force.
If $(Y,Z,R)$ is a solution to \textnormal{RBSDE} \eqref{2marca1}--\eqref{eq.2903}, then 
\begin{equation*}
\begin{split}
Y_\theta&=\mathop{\mathrm{ess\,inf}}_{\sigma\ge \theta,  G\in\mathcal F_\sigma}\mathop{\mathrm{ess\,sup}}_{\tau\ge \theta,H\in\mathcal{F}_{\tau}}\mathbb{E}^f_{\theta,\tau\wedge\sigma}J(\tau,H;\sigma,G)\\
&=\mathop{\mathrm{ess\,sup}}_{\tau\ge\theta, H\in\mathcal{F}_{\tau}}\mathop{\mathrm{ess\,inf}}_{\sigma\ge\theta, G\in\mathcal{F}_{\sigma}}\mathbb{E}^f_{\theta,\tau\wedge\sigma}J(\tau,H;\sigma,G)
\end{split}
\end{equation*}
for any stopping time $\theta\le T$.
}
\end{minipage}
\end{center}
\vspace{1cm}
In other words, we show that $Y$ is a value process in an {\em extended nonlinear  Dynkin game};  "nonlinear" since we consider the 
nonlinear expectation,
and "extended" since players may 
change payoffs $L,U$ on  sets $H^c,G^c$, respectively, which extends the set of their strategies
(in the classical Dynkin games the players are not allowed to choose sets $G,H$). Observe that the above extended nonlinear Dynkin game 
reduces to the  nonlinear Dynkin game provided $L$ and $U$ are right-continuous. We prove however a stronger result.
\vspace*{0.20cm}
\begin{center}
\begin{minipage}[c][2,55cm][t]{0,85\textwidth}
{\em 
\textbf{Theorem 4.} Assume that \textnormal{(A1)--(A4), (Z)} are in force.
Moreover, suppose  that  $L$ is right upper semicontinuous and $U$ is right lower semicontinuous.
If $(Y,Z,R)$ is a solution to \textnormal{RBSDE} \eqref{2marca1}--\eqref{eq.2903}, then 
\begin{equation*}
\begin{split}
Y_\theta&=\mathop{\mathrm{ess\,inf}}_{\sigma\ge \theta}\mathop{\mathrm{ess\,sup}}_{\tau\ge \theta}\mathbb{E}^f_{\theta,\tau\wedge\sigma}J(\tau,\Omega;\sigma,\Omega)\\
&=\mathop{\mathrm{ess\,sup}}_{\tau\ge\theta}\mathop{\mathrm{ess\,inf}}_{\sigma\ge\theta}\mathbb{E}^f_{\theta,\tau\wedge\sigma}J(\tau,\Omega;\sigma,\Omega).
\end{split}
\end{equation*}
for any stopping time $\theta\le T$.
}
\end{minipage}
\end{center}
\vspace{1.5cm}
Thus,  $Y$ represents  the  value process in a  nonlinear  Dynkin game provided $L, U$ are sufficiently  regular as mentioned above. The above result was achieved by Bayraktar and Yao  in \cite{BY} for  continuous barriers $L,U$ satisfying \eqref{cond1} and under the following additional conditions: $\mathbb E\sup_{t\le T}|L_t|+\mathbb E\sup_{t\le T}|U_t|<\infty$,   generator $f$ admits  the linear
growth with respect to $Y$-variable, i.e.   $|f(t,y,0)|\le g_t+\psi|y|$ for some $\psi\ge 0$. Note that in the present paper growth of $f$ 
with respect to $Y$-variable is subject to no restriction. 

Finally, we show that further regularity assumptions on barriers $L,U$ allow one to indicate saddle points for nonlinear Dynkin games.
For    any stopping time $\theta\le T$ set:
\begin{equation}\label{dyn14intr}
\tau^*_{\theta}:=\inf\{t\ge\theta,\,Y_t=L_t\}\wedge T;\quad\sigma^*_{\theta}:=\inf\{t\ge\theta,\,Y_t=U_t\}\wedge T
\end{equation}
and
\begin{equation}\label{dyn15intr}
\begin{split}
&\bar{\tau}_{\theta}:=\inf\{t\ge\theta,\,R^{+}_t>R^{+}_{\theta}\}\wedge T;\quad\bar{\sigma}_{\theta}:=\inf\{t\ge\theta,\,R^{-}_t>R^{-}_{\theta}\}\wedge T.
\end{split}
\end{equation}

\vspace*{0.20cm}
\begin{center}
\begin{minipage}[c][2,65cm][t]{0,85\textwidth}
{\em 
\textbf{Theorem 5.} Assume that \textnormal{(A1)--(A4), (Z)} are in force.
Moreover, suppose  that  $L$ is  upper semicontinuous and $U$ is  lower semicontinuous.
Then
\begin{equation*}
\begin{split}
\mathbb{E}^f_{\theta,\tau^*_\theta\wedge\sigma^*_\theta}J(\tau^*_\theta,\Omega;\sigma^*_\theta,\Omega)&=
\mathbb{E}^f_{\theta,\bar\tau_\theta\wedge\bar\sigma_\theta}J(\bar\tau_\theta,\Omega;\bar\sigma_\theta,\Omega)\\
&=\mathop{\mathrm{ess\,inf}}_{\sigma\ge \theta}\mathop{\mathrm{ess\,sup}}_{\tau\ge \theta}\mathbb{E}^f_{\theta,\tau\wedge\sigma}J(\tau,\Omega;\sigma,\Omega).
\end{split}
\end{equation*}
for any stopping time $\theta\le T$.
}
\end{minipage}
\end{center}

\vspace{1cm}
{\bf Proof techniques and relations of main results  to the existing literature.}  First, note that 
Mokobodzki's condition, additionally to  (WM),  requires from   the semimartingale $X$, lying between the barriers,
some integrability of its   finite variation and  martingale part
(depending on the authors you may find different integrability conditions for  the process $X$, nonetheless  it  is always  assumed  to be at least  the  difference 
of positive supermartingales).
This additional requirement is the reason why the complete separation  condition \eqref{cond1} does not imply  Mokobodzki's condition.
The fact that \eqref{cond1} implies {\em weak Mokobodzki's condition} is an easy calculation and may be found e.g. in \cite{Topolewski} and \cite{HH1} (in the case of continuous barriers).
Reflected BSDEs  with condition \eqref{cond1} imposed on the barriers have been  considered in many papers
(see \cite{BY,BL,EHW,HH,HHO,HHd,Hassairi,Topolewski}). In the papers \cite{BL}, \cite{HH} and \cite{HHO}, $L^2$-data and sublinear growth of the generator with respect to $Y$-variable
are required. In \cite{HHd} authors considered bounded data and  continuous generator with  quadratic-growth  with respect to $Z$-variable. 
Some results for RBSDEs with generators subject to sublinear growth   with respect to $Y$-variable are described also in  \cite{BY,Hassairi} ($L^1$-data) and \cite{EHW} ($L^p$-data, $p\in(1,2)$). $L^1$-data and generator being merely monotone and continuous with respect to $Y$-variable were considered in \cite{Topolewski}. 

In all the mentioned  papers (besides \cite{HHd}), the method of local solutions
and pasting local solutions, introduced by Hamad\`ene and Hassani in \cite{HH}, has been applied to achieve  the existence for underlying RBSDEs. 
This method is rather complicated,  and this is perhaps the reason why the development of theory of RBSDEs with barriers satisfying complete 
separation condition   is far from being satisfactory. The second drawback of the method is that it is based on the penalization
scheme which is not available for RBSDEs with  optional barriers.
 In \cite{kl}, the  author proposed a different method which applies to
RBSDEs with barriers satisfying even more general  than \eqref{cond1} {\em weak Mokobodzki's condition} (WM).
We call this method  {\em localization procedure}.  The advantage of the method is its simplicity and wide applicability.  
The method is based on the following simple observation: for any   {\em chain} $(\tau_k)$, i.e.  non-decreasing 
sequence of stopping times satisfying
\[
P(\tau_k<T,\, k\ge 1)=0,
\]
we have
\begin{equation*}
\begin{split}
&(Y,Z,R)\, \text{ solves  RBSDE}^T(\xi,f,L,U)\\
&\quad\quad\quad\quad\quad\quad\quad\text{iff}\\
&(Y,Z,R)\, \text{ solves  RBSDE}^{\tau_k}(Y_{\tau_k},f,L,U),\, \,k\ge 1.
\end{split}
\end{equation*}
The   method consists of finding a proper regular approximation $(Y^n)$, on the whole interval $[0,T]$, of a potential solution $Y$
of a given problem (by "proper" we mean an approximation which does not blow up when passing to the  limit). 
The terms of approximating sequence may solve  BSDEs or RBSDEs of the generic form
\[
Y^n_t=\xi_n+\int_t^Tf_n(r,Y^n_r,Z^n_r)\,dr+R^n_T-R^n_t-\int_t^TZ^n_r\,dB_r\quad t\in [0,T],
\]
with suitable chosen $\xi_n,f_n,R^n$.  In the first step one shows that $(Y^n)$  converges to a process $Y$.
After   that, we show that  $Y$ is the {\em main part} of a solution to RBSDE$^{\tau_k}(Y_{\tau_k},f,L,U)$ for each $k\ge 1$.
Since $(\tau_k)$ is a {\em chain}, we conclude  that  $Y$ is the {\em main part} of a solution to  RBSDE$^T(\xi,f,L,U)$ .

By using {\em localization procedure} in \cite{kl}, the first author of the present paper was able to provide an existence result for 
RBSDEs with merely c\`adl\`ag barriers of class (D) satisfying (WM), $L^1$-data,  and  generator being
continuous and non-increasing with respect to $Y$-variable (with no restrictions on the growth of the generator with respect to $Y$-variable).   

As far as we know the only papers in the literature concerned with  RBSDEs of the form  \eqref{2marca1} with non-c\`adl\`ag barriers satisfying (WM) are \cite{KRz,MM0}.
In \cite{KRz} RBSDEs on a general filtered space are  studied under (WM) but with $f$ independent of $Z$-variable.
In The generalization of \eqref{cond1} to the case of l\`adl\`ag barriers was presented in \cite{MM0} the authors considered l\`adl\`ag barriers
and stochastic Lipschitz generator $f$ (on the Brownian-Poisson filtration).

As to the nonlinear Dynkin games, to the best of  our knowledge,  there are only  few papers  in the literature:  \cite{DQS,DQS2,DQS3,GIOQ2,GQ} - all with  $L^2$-data and Lipschitz generator -  and \cite{K7,KRz} - with $L^1$-data and continuous and monotone generator with respect to $Y$-variable and independent of $Z$-variable.


\medskip
{\bf Comments on the related literature.}
Reflected backward stochastic differential equations with two barriers have been  introduced by Cvitani\'{c} and Karatzas in \cite{CK} 
as a generalization of backward stochastic differential equations introduced by Pardoux and Peng in \cite{PP} (analogous results for one reflecting barrier, i.e. in  case $U\equiv\infty$, have been  presented for the first time by El Karoui et al. in \cite{EKPPQ}). 
In \cite{CK} the authors considered \eqref{2marca1} with barriers being continuous processes satisfying {\em  Mokobodzki's condition} i.e. there exists $L_t\le X_t\le U_t$, $t\in[0,T]$, such that $X=X^1-X^2$ and $X^i$ is a  positive supermartingale satisfying $\mathbb E\sup_{t\le T}|X^i|^2_t<\infty,\, i=1,2$. Moreover, they assumed that data are $L^2$-integrable (i.e. $\sup_{t\le T}|L_t|, \sup_{t\le T}|U_t|, |\xi|, \int_0^T|f(r,0,0)|\,dr$
have second moments) and $f$ is  Lipschitz continuous with respect to $(Y,Z)$-variable (uniformly in $(\omega,t)$). Under these assumptions a solution to \eqref{2marca1}
has been defined in \cite{CK} as a  triple $(Y,Z,R)$  of $\mathbb F$-progressively measurable processes such that $Y$ is  continuous,  and $R$ is a continuous  finite variation process, with $R_0=0$, satisfying the minimality condition of the form
\[
\int^T_0(Y_r-L_r)\,dR^+_r=\int^T_0(U_r-Y_r)\,dR^-_r=0,
\]
where $R=R^+-R^-$ is the Jordan decomposition of $R$. Observe that with continuous  $Y,L,U,R$, condition  \eqref{eq.2903} reduces to  the above condition.

BSDEs and Reflected BSDEs are of great interest to scientists because of their numerous applications in various fields of mathematics and problems (e.g.  partial differential equations, integro-differential equations,  variational inequalities, optimization theory, control theory, mathematical finance etc., see \cite{Crepey,PR,Zhang} and the references therein). 
Over the past two decades, many interesting results have been obtained regarding RSBDEs.
In particular, numerous existence results for RBSDEs, which strengthen the result  of  \cite{CK}  by weakening  assumptions on generator $f$, filtration $\mathbb{F}$, barriers $L,U$ and horizon time $T$, have been provided.

Despite of intensive research, until  2016, only RBSDEs with c\`adl\`ag barriers were  considered  in the literature. 
With the work by Grigorova et al. in \cite{GIOOQ} there  was a change in this regard and papers on less regular barriers began to appear.
Equations of that type with $L^2$-data and Lipschitz generator were studied in \cite{MM} (Brownian filtration), in \cite{GIOOQ,GIOQ2} (Brownian-Poisson filtration) and \cite{BO,BO2,GIOQ} (general filtration). RBSDEs with optional barriers and $L^1$-data were considered only in \cite{KRzS,KRzS2}, in the case of Brownian filtration. Results on optional barriers, $L^1$-data and possibly infinite horizon time were presented in \cite{KRz} but with $f$ independent of $Z$-variable.
The case of $L^2$-data and $f$ being stochastic Lipschitz driver was presented in \cite{MM2} (Brownian-Poisson filtration) and in \cite{Marzogue,MM0} (general filtration).

\section{Basic notation}\label{roz2}

We say that a function $y:[0,T]\to\mathbb{R}^d$ is regulated on $[0,T]$ if for any $t\in[0,T)$, there exists the limit $y_{t+}:=\lim_{u\downarrow t}y_u$ and for any $s\in(0,T]$ there exists the limit $y_{s-}:=\lim_{u\uparrow s}y_u$. For any regulated function $y$ on $[0,T]$ we define $\Delta^+y_t:=y_{t+}-y_t$, $t\in[0,T)$ and $\Delta^-y_s:=y_s-y_{s-}$, $s\in(0,T]$.

For $x\in\mathbb{R}^d$ by $|x|$ we denote the euclidean norm. As mentioned in Section \ref{wstep}, $\mathcal{T}$ stands for the set of all stopping times taking values in $[0,T]$. What is more, for $\alpha,\beta\in\mathcal{T}$, $\mathcal{T}_{\alpha,\beta}:=\{\tau\in\mathcal{T},\,\alpha\le\tau\le\beta\}$, $\mathcal{T}_{\alpha}:=\mathcal{T}_{\alpha,T}$, $\mathcal{T}^{\beta}:=\mathcal{T}_{0,\beta}$.

Let $\alpha,\beta\in\mathcal{T}$, $\alpha\le\beta$, and $p\ge 1$. By $\mathcal{S}^p_{\mathbb{F}}(\alpha,\beta)$ we denote all $\mathbb{F}$-progessively measurable, $\mathbb{R}$-valued processes $Y=(Y_t)_{t\in[0,T]}$ such that
\[
||Y||_{\mathcal{S}^p_{\mathbb{F}}(\alpha,\beta)}:=\Big(\mathbb{E}\sup_{\alpha\le t\le\beta}|Y_t|^p\Big)^{\frac{1}{p}}<\infty.
\]
$\mathcal{M}_{loc}(\alpha,\beta)$ is the space of all $\mathbb{F}$-local martingales on $[[\alpha,\beta]]$. Let $q\ge 1$. By $L^{p,q}_{\mathbb{F}}(\alpha,\beta)$ we denote the set of all $\mathbb{F}$-progressively measurable, $\mathbb{R}$-valued processes $X=(X_t)_{t\in[0,T]}$ such that
\[
||X||_{L^{p,q}_{\mathbb{F}}(\alpha,\beta)}:=\Bigg(\mathbb{E}\Big(\int^{\beta}_{\alpha}|X_r|^p\,dr\Big)^{\frac{q}{p}}\Bigg)^{\frac{1}{p}}<\infty.
\]
$L^p_{\mathbb{F}}(\alpha,\beta)$ is a shorthand for $L^{p,p}_{\mathbb{F}}(\alpha,\beta)$.

Let $\mathcal{G}\subset\mathcal{F}$. $L^p(\mathcal{G})$ is the set of all $\mathcal{G}$-measurable random variables $X$ such that
\[
||X||_{L^p}:=\Big(\mathbb{E}|X|^p\Big)^{\frac{1}{p}}<\infty.
\]
By $\mathcal{H}_{\mathbb{F}}(\alpha,\beta)$, we denote the space of all $\mathbb{F}$-progessively measurable, $\mathbb{R}^d$-valued processes $Z=(Z_t)_{t\in[0,T]}$ such that
\[
\int^{\beta}_{\alpha}|Z_r|^2\,dr<\infty\quad P\mbox{-a.s.}
\]
$\mathcal{H}^s_{\mathbb{F}}(\alpha,\beta)$, $s>0$, is a  subspace of $\mathcal{H}_{\mathbb{F}}(\alpha,\beta)$ consisting of $Z$ satisfying
\[
\mathbb{E}\Big(\int^{\beta}_{\alpha}|Z_r|^2\,dr\Big)^{\frac{s}{2}}<\infty.
\] 
We say that $\mathbb{F}$-progessively measurable process $X=(X_t)_{t\in[0,T]}$ is of class (D) on $[[\alpha,\beta]]$ if the family $\{X_{\tau},\,\tau\in\mathcal{T}_{\alpha,\beta}\}$ is uniformly integrable. By $\mathcal{D}^2_{\mathbb{F}}(\alpha,\beta)$ we denote the set of all $\mathbb{F}$-progressively measurable, $\mathbb{R}$-valued processes $Y=(Y_t)_{t\in[0,T]}$ such that $|Y|^2$ is of class (D) on $[[\alpha,\beta]]$. We equip $\mathcal D^2_{\mathbb{F}}(\alpha,\beta)$ with the norm
\[
||Y||_{\mathcal{D}^2(\alpha,\beta)}:=\Big(\sup_{\sigma\in\mathcal{T}_{\alpha,\beta}}\mathbb E|Y_{\sigma}|^2\Big)^{\frac{1}{2}}.
\]

A sequence $(\tau_k)_{k\ge 1}\subset\mathcal{T}_{\alpha,\beta}$ is called a chain on $[[\alpha,\beta]]$ if
\[
\forall_{\omega\in\Omega}\exists_{n\in\mathbb{N}}\forall_{k\ge n}\,\tau_k(\omega)=\beta(\omega).
\]

By $\mathcal{V}_{\mathbb{F}}(\alpha,\beta)$ (resp. $\mathcal{V}^+_{\mathbb{F}}(\alpha,\beta)$) we denote a space of $\mathbb{F}$-progessively measurable, $\mathbb{R}$-valued processes $V=(V_t)_{t\in[0,T]}$ with finite variation (resp. nondecreasing) on $[[\alpha,\beta]]$ and $\mathcal{V}_{0,\mathbb{F}}(\alpha,\beta)$ (resp. $\mathcal{V}^+_{0,\mathbb{F}}(\alpha,\beta)$) is a subspace of  $\mathcal{V}_{\mathbb{F}}(\alpha,\beta)$ (resp. $\mathcal{V}^+_{\mathbb{F}}(\alpha,\beta)$) 
consisting of  processes $V$ such that $V_{\alpha}=0$. $\mathcal{V}^p_{\mathbb{F}}(\alpha,\beta)$ (resp. $\mathcal{V}^{+,p}_{\mathbb{F}}(\alpha,\beta)$) is the set of all $V\in\mathcal{V}_{\mathbb{F}}(\alpha,\beta)$ (resp. $V\in\mathcal{V}^+_{\mathbb{F}}(\alpha,\beta)$) such that $E|V|^p_{\alpha,\beta}<\infty$, where $|V|_{\alpha,\beta}$ denotes the total variation of $V$ on $[[\alpha,\beta]]$.

Let $V\in\mathcal{V}_{\mathbb{F}}(0,T)$. By $V^*$ we denote the c\`adl\`ag part of the process $V$, i.e.
\[
V^*_t=V_t-\sum_{0\le r<t}\Delta^+V_r.
\]

Throughout the paper all relations between random variables are supposed to hold $P$-a.s. For processes $X^1=(X^1_t)_{t\in[0,T]}$ and $X^2=(X^2_t)_{t\in[0,T]}$ we write $X^1\le X^2$ if $X^1_t\le X^2_t$, $t\in[0,T]$, $P$-a.s.

Let $V^1,V^2\in\mathcal{V}_{0,\mathbb{F}}(0,\tau)$. We write $dV^1\le dV^2$, if $dV^{1,*}\le dV^{2,*}$ and $\Delta^+V^1\le\Delta^+V^2$ on $[0,\tau]$. 

For an $\mathbb{F}$-optional process $X=(X_t)_{t\in[0,T]}$ we set  $\overrightarrow{X}_s=\limsup_{r\uparrow s}X_r$, $\underrightarrow{X}_s=\liminf_{r\uparrow s}X_r$, $s\in(0,T]$ and $\overleftarrow{X}_s=\limsup_{r\downarrow s}X_r$, $\underleftarrow{X}_s=\liminf_{r\downarrow s}X_r$, $s\in[0,T]$, $s\in[0,T)$.

\section{Backward SDEs}\label{BSDE}\label{roz3}

Let $p\ge 1$. We shall need the following hypotheses:
\begin{enumerate}
\item[(H1)] there is $\lambda\ge0$ such that
$|f(t,y,z)-f(t,y,z')|\le\lambda|z-z'|$ for $t\in[0,T]$, $y\in\mathbb{R}$, $z,z'\in\mathbb{R}^d$,
\item[(H2)] there is $\mu\in\mathbb{R}$ such that
$(y-y')(f(t,y,z)-f(t,y',z))\leq\mu(y-y')^2$ for $t\in[0,T]$, $y,y'\in\mathbb{R}$, $z\in\mathbb{R}^d$,
\item[(H3)] for every $(t,z)\in[0,T]\times\mathbb{R}^d$ the mapping $\mathbb{R}\ni y\rightarrow f(t,y,z)$ is continuous,
\item[(H4)] $\int^T_0 |f(r,y,0)|\,dr<\infty$ for every $y\in\mathbb{R}$,
\item[(H5)] $\xi\in L^p(\mathcal{F}_T)$, $f(\cdot,0,0)\in L^{1,p}_{\mathbb F}(0,T)$,
\end{enumerate}





Let $\alpha,\beta\in\mathcal{T}$, $\alpha\le\beta$, and $\hat{\xi}\in\mathcal{F}_{\beta}$. 

\begin{definition}
We say that a pair $(Y,Z)$ of $\mathbb{F}$-adapted processes is a solution to  backward stochastic differential equation
on the interval $[[\alpha,\beta]]$ with right-hand side $f$ and terminal value $\hat{\xi}$ (BSDE$^{\alpha,\beta}(\hat{\xi},f)$ for short) if
\begin{enumerate}
\item[(a)] $Y$ is a continuous process  and $Z\in\mathcal{H}_{\mathbb{F}}(\alpha,\beta)$,
\item[(b)] $\int^{\beta}_{\alpha}|f(r,Y_r,Z_r)|\,dr<\infty$,
\item[(c)] $Y_t=\hat{\xi}+\int^{\beta}_t f(r,Y_r,Z_r)\,dr-\int^{\beta}_t Z_r\,dB_r$, $t\in[\alpha,\beta]$.
\end{enumerate}
\end{definition}

Let $V\in\mathcal{V}_{0,\mathbb F}(\alpha,\beta)$. 

\begin{definition}
We say that a pair $(Y,Z)$ of $\mathbb{F}$-adapted processes is a solution to  backward stochastic differential equation
on the interval $[[\alpha,\beta]]$ with right-hand side $f+dV$ and terminal value $\hat{\xi}$ (BSDE$^{\alpha,\beta}(\hat{\xi},f+dV)$ for short) if
$(Y-V,Z)$ is a solution to BSDE$^{\alpha,\beta}(\hat{\xi},f_V)$, where $f_V(t,y,z)=f(t,y+V_t,z)$.
\end{definition}

Let us adopt the shorthand BSDE$^{\beta}:=$BSDE$^{0,\beta}$.

The following results follow from  \cite[Proposition 3.2, Theorem 4.2]{bdh}.

\begin{theorem}\label{5grudnia3}
 Assume that \textnormal{(H1)--(H4)} are in force. Suppose that \textnormal{(H5)} holds with  $p>1$. 
 Then the following assertions hold. 
 \begin{enumerate}
 \item[(i)] There exists a  solution $(Y,Z)\in\mathcal{S}^p_{\mathbb{F}}(0,T)\times \mathcal H^p_{\mathbb{F}}(0,T)$ to \textnormal{BSDE}$^T(\xi,f)$.
 \item[(ii)] There exists at most one solution  $(Y,Z)$ to \textnormal{BSDE}$^T(\xi,f)$ such that $Y\in \mathcal S^p_{\mathbb F}(0,T)$.
 \end{enumerate} 
\end{theorem}

\begin{proposition}\label{13stycznia19}
 Assume that \textnormal{(H5)}, with $p>1$, and \textnormal{(H1),(H2)} are satisfied. Let $(Y,Z)$ be a solution to \textnormal{BSDE}$^T(\xi,f)$ such that $Y\in\mathcal{S}^p_{\mathbb{F}}(0,T)$. Then there exists $c>0$, depending only on $\mu,\lambda,T,p$, such that
\begin{equation}
\label{13stycznia20}
\begin{split}
\mathbb E\Big[\sup_{0\le t\le T}|Y_t|^p+\Big(\int^T_0 |Z_r|^2\,dr\Big)^{\frac{p}{2}}&+\Big(\int^T_0 |f(r,Y_r,Z_r)|\,dr\Big)^p\Big]\\&
\le c\mathbb E\Big[|\xi|^p+\Big(\int^T_0 |f(r,0,0)|\,dr\Big)^p\Big].
\end{split}
\end{equation}
\end{proposition}

In case (H5) is satisfied with $p=1$, we  shall need for the existence and uniqueness of solutions to BSDEs  additional hypothesis.
\begin{enumerate}
\item[(Z)] There exists an $\mathbb{F}$-progressively measurable process $g\in L^1_{\mathbb F}(0,T)$ and $\gamma\ge 0$, $\kappa\in [0,1)$ such that
\begin{align*}
|f(t,y,z)-f(t,y,0)|\le\gamma(g_t+|y|+|z|)^\kappa,\quad t\in [0,T], \,y\in\mathbb{R},\,z\in\mathbb{R}^d.
\end{align*}
\end{enumerate}

\begin{remark}
\label{remark.z}
Condition (Z) says that
driver $f$ is allowed to have at most  sublinear growth with respect to $z$-variable.
A typical example of a driver satisfying (H1)--(H5), (Z) is of the following  form
\[
f(t,y,z):= f_0(t,y)+ b(y) (1+|z|)^\kappa,
\]
where $f_0$ satisfies (H2)--(H5), $b$ is continuous, non-increasing and bounded, and $\kappa\in (0,1)$.

Observe that, in general, under merely (H1)--(H4) and  (H5) with $p=1$, we cannot expect the existence of a  solution $(Y,Z)$
to BSDE$^T(\xi,f)$ with positive $\xi$ such that $Y$ is positive and of class (D). Indeed, assume that $(Y,Z)$ is a solution to  the following BSDE
\[
Y_t=\xi+\int_t^T Z_r\,dr-\int_t^TZ_r\,dB_r,\quad t\in [0,T],
\]
with positive $\xi\in L^1(\mathcal{F}_T)$, and $Y$ is positive of class (D).  Let $(\tau_k)$ be
chain such that $Y\in \mathcal S^2(0,\tau_k), Z\in \mathcal H^2_{\mathbb F}(0,\tau_k),\, k\ge 1$. Then,
by It\^o's formula
\[
Y_0=\mathbb E\Big[Y_{\tau_k}\exp(-\frac{\tau_k}{2}+B_{\tau_k})\Big]
\]
Therefore, by applying Fatou's lemma, we find that
\[
e^TY_0\ge\mathbb E[\xi \exp(B_T)].
\]
If the above inequality  was true for any positive $\xi\in L^1(\mathcal{F}_T)$, then $\exp(B_T)$ would be bounded, a contradiction.
\end{remark}


\begin{theorem}\label{12sierpnia1}
Assume that \textnormal{(H1)-(H4), (Z)} are in force. Moreover, assume that \textnormal{(H5)} is satisfied with $p=1$. 
 Then the following assertions hold.
 \begin{enumerate}
 \item[(i)]  There exists a  solution $(Y,Z)$ of \textnormal{BSDE}$^T(\xi,f)$ such that $Y$ is of class \textnormal{(D)} and $Z\in\mathcal{H}^s_{\mathbb{F}}(0,T)$, $s\in(0,1)$.
 \item[(ii)] There exists at most one solution  $(Y,Z)$ to \textnormal{BSDE}$^T(\xi,f)$ such that $Y$ is of class  \textnormal{(D)}.
 \end{enumerate} 
\end{theorem}
\begin{proof}
The assertion (i) follows from \cite[Theorem 6.3]{bdh}. As to  (ii), by \cite[Theorem 6.2]{bdh},
there exists at most one solution  $(Y,Z)$ to \textnormal{BSDE}$^T(\xi,f)$ such that $Y$ is of class  \textnormal{(D)}
and $Z\in\mathcal{H}^s_{\mathbb{F}}(0,T)$, $s\in(0,1)$. So, it is enough to show that if 
$(Y,Z)$ is a solution  to \textnormal{BSDE}$^T(\xi,f)$ such that $Y$ is of class  \textnormal{(D)}, then $Z\in\mathcal{H}^s_{\mathbb{F}}(0,T)$, $s\in(0,1)$.
This follows at once from \cite[Remark 2.1]{KRzS2} and \cite[Lemma 3.1]{bdh}.
\end{proof}

\section{Reflected BSDEs with two optional barriers under  Mokobodzki's condition}\label{roz4}

In this section we assume that processes $L$ and $U$ are merely $\mathbb{F}$-optional. 
Let $\alpha,\beta\in\mathcal{T}$, $\alpha\le\beta$, and $\hat{\xi}\in\mathcal{F}_{\beta}$ such that $L_{\beta}\le\hat{\xi}\le U_{\beta}$. 

\begin{definition}
We say that a triple $(Y,Z,R)$ of $\mathbb{F}$-adapted processes is a solution to  reflected backward stochastic differential equation
on the interval $[[\alpha,\beta]]$ with right-hand side $f$, terminal value $\hat{\xi}$, lower barrier $L$ and upper barrier $U$ (RBSDE$^{\alpha,\beta}(\hat{\xi},f,L,U)$ for short) if
\begin{enumerate}
\item[(a)] $Y$ is a regulated process and $Z\in\mathcal{H}_{\mathbb{F}}(\alpha,\beta)$,
\item[(b)] $R\in\mathcal{V}_{0,\mathbb{F}}(\alpha,\beta)$, $L_t\le Y_t\le U_t$, $t\in[\alpha,\beta]$, and
\begin{equation*}
\begin{split}
&\int^{\beta}_{\alpha}(Y_{r-}-\overrightarrow{L}_r)\,dR^{*,+}_r+\sum_{\alpha\le r<\beta}(Y_r-L_r)(\Delta^+R_r)^+\\&
=\int^{\beta}_{\alpha}(\underrightarrow{U}_r-Y_{r-})\,dR^{*,-}_r+\sum_{\alpha\le r<\beta}(U_r-Y_r)(\Delta^+R_r)^-=0,
\end{split}
\end{equation*}
where $R^*=R^{*,+}-R^{*,-}$ is the Jordan decomposition of $R^*$,
\item[(c)] $\int^{\beta}_{\alpha}|f(r,Y_r,Z_r)|\,dr<\infty$,
\item[(d)] $Y_t=\hat{\xi}+\int^{\beta}_t f(r,Y_r,Z_r)\,dr+R_{\beta}-R_t-\int^{\beta}_t Z_r\,dB_r$, $t\in[\alpha,\beta]$.
\end{enumerate}
\end{definition}

In what follows we refer to condition (b) as the {\em minimality condition}.

Let us adopt the shorthand RBSDE$^{\beta}:=$RBSDE$^{0,\beta}$.

We consider the following  condition, which we call
{\em  strong Mokobodzki's condition}:

\begin{enumerate}
\item[(H6)] there exists a process $X\in\mathcal{M}_{loc}(0,T)+\mathcal{V}^p_{\mathbb F}(0,T)$ such that
 $L\le X\le U$, $X\in \mathcal S^p(0,T)$ and $f(\cdot,X,0)\in L^{1,p}_{\mathbb F}(0,T)$.
\end{enumerate}

Assume that $L_T\le\xi\le U_T$. The following result has been proven in \cite[Proposition 3.2,Theorem 3.9]{KRzS2}.

\begin{theorem}
\label{18listopada1}
Let $p>1$.  Assume  \mbox{\rm(H1)-(H5)}. 
\begin{enumerate}
\item[(i)] There exists at most one solution $(Y,Z,R)$ to \textnormal{RBSDE}$^T(\xi,f,L,U)$  
such that $Y\in\mathcal{S}^p_{\mathbb{F}}(0,T)$.
\item[(ii)] There exists a  solution $(Y,Z,R)$ to \textnormal{RBSDE}$^T(\xi,f,L,U)$ 
such that $Y\in\mathcal{S}^p_{\mathbb{F}}(0,T)$, $Z\in\mathcal{H}^p_{\mathbb{F}}(0,T)$ and $R\in\mathcal{V}^p_{0,\mathbb{F}}(0,T)$
if and only if  \mbox{\rm(H6)} holds.
\end{enumerate}

\end{theorem}

In the case of $p=1$, we consider the following version of 
{\em  strong Mokobodzki's condition}:

\begin{enumerate}
\item[(H6*)] there exists a process $X\in\mathcal{M}_{loc}(0,T)+\mathcal{V}_{\mathbb{F}}^1(0,T)$ such that $X$ is of class (D), $L\le X\le U$ and $f(\cdot,X,0)\in L^1_{\mathbb{F}}(0,T)$.
\end{enumerate}

The following result has been proven in \cite[Theorem 3.8]{KRzS2}.

\begin{theorem}\label{28wrzesnia1}
Let $p=1$.  Assume  \textnormal{(H1)-(H5), (H6*), (Z)}. Then there exists a unique solution $(Y,Z,R)$ of \textnormal{RBSDE}$^T(\xi,f,L,U)$ such that $Y$ is of class {\rm (D)}, $Z\in\mathcal{H}^q_{\mathbb{F}}(0,T)$, $q\in(0,1)$ and $R\in\mathcal{V}^1_{0,\mathbb{F}}(0,T)$.
\end{theorem}

\section{Nonlinear expectation}\label{roz5}
Let $p\ge 1$. Throughout this section, we assume that  either  $p=1$ and (H1)-(H5), (Z) are in force
or $p>1$ and (H1)-(H5) are in force. Let $\alpha,\beta\in\mathcal{T}$, $\alpha\le\beta$. We define the operator
\[
\mathbb{E}^{(1),f}_{\alpha,\beta}:L^1(\mathcal{F}_{\beta})\to L^1(\mathcal{F}_{\alpha}),
\]
by letting $\mathbb{E}^{(1),f}_{\alpha,\beta}(\xi):=Y_{\alpha}$,
where $(Y,Z)$ is a  solution to BSDE$^{\beta}(\xi,f)$ such that $Y$ is of class (D).
By Theorem \ref{12sierpnia1}, the operator  $\mathbb{E}^{(1),f}_{\alpha,\beta}$ is well defined under conditions (H1)-(H5), (Z). By  Theorem \ref{5grudnia3}, under (H1)--(H5) (with $p>1$), we  may  define the  operator
\[
\mathbb{E}^{(p),f}_{\alpha,\beta}:L^p(\mathcal{F}_{\beta})\to L^p(\mathcal{F}_{\alpha}),
\]
with $\mathbb{E}^{(p),f}_{\alpha,\beta}(\xi):=Y_{\alpha}$, where $(Y,Z)$ is a  solution to BSDE$^{\beta}(\xi,f)$ such that $Y\in\mathcal S^p_{\mathbb{F}}(0,\beta)$.
Finally, we define operator 
\[
\mathbb{E}^{f}_{\alpha,\beta}:L^1(\mathcal{F}_{\beta})\to L^1(\mathcal{F}_{\alpha}),
\]
by letting

\begin{equation}
\mathbb E^{f}_{\alpha,\beta}(\xi):=
\begin{cases}
\mathbb{E}^{(1),f}_{\alpha,\beta}(\xi) & \text{for } \xi\in L^1(\mathcal{F}_{\beta})\setminus \bigcup_{p>1}L^p(\mathcal{F}_{\beta})\\
\mathbb{E}^{(p),f}_{\alpha,\beta}(\xi) & \text{for } \xi  \in \bigcup_{p>1}L^p(\mathcal{F}_{\beta}).
\end{cases}
\end{equation}

We say that a process $X$ of class (D) is an $\mathbb{E}^f$-supermartingale (resp. $\mathbb{E}^f$-submartingale) on $[[\alpha,\beta]]$, if $\mathbb{E}^f_{\sigma,\tau}(X_{\tau})\le X_{\sigma}$ (resp. $\mathbb{E}^f_{\sigma,\tau}(X_{\tau})\ge X_{\sigma}$) for every $\sigma,\tau\in\mathcal{T}_{\alpha,\beta}$, $\sigma\le\tau$. $X$ is an $\mathbb{E}^f$-martingale on $[[\alpha,\beta]]$, if $X$ is at the same time an $\mathbb{E}^f$-supermartingale and an $\mathbb{E}^f$-submartingale on $[[\alpha,\beta]]$.

\begin{remark}\label{7grudnia8}
Note that the process $Y$ of class \textnormal{(D)} is an $\mathbb{E}^f$-martingale on $[[\alpha,\beta]]$ if and only if $Y$ is indistinguishable from the first component of the solution to BSDE$^{\alpha,\beta}(Y^\beta,f)$ on $[[\alpha,\beta]]$.
Thus, in order to prove that   $Y$ is an $\mathbb{E}^f$-martingale on $[[\alpha,\beta]]$, it suffices to show that $Y_{\sigma}=\mathbb{E}^f_{\sigma,\beta}(Y_{\beta})$, for any  $\sigma\in\mathcal{T}_{\alpha,\beta}$.
\end{remark}

\begin{proposition}\label{nonlinprop}
Let  $\alpha,\beta\in\mathcal{T}$, $\alpha\le\beta$.
\begin{enumerate}
\item[(i)] Let $\xi\in L^p(\mathcal{F}_{\beta})$ and let $V$ be an $\mathbb{F}$-adapted, finite variation process such that $V_{\alpha}=0$. Let $(X,H)$ be a solution to \textnormal{BSDE}$^{\alpha,\beta}(\xi,f+dV)$ such that $X$ is of class \textnormal{(D)}, in case $p=1$, and $X\in \mathcal S^p_{\mathbb F}(\alpha,\beta)$, in case $p>1$.
If $V$ (resp. $-V$) is increasing, then $X$ is $\mathbb{E}^f$-supermartingale (resp. $\mathbb{E}^f$-submartingale) on $[[\alpha,\beta]]$.
\item[(ii)] If $\xi_1,\xi_2\in L^p(\mathcal{F}_{\beta})$ and $\xi_1\le\xi_2$, then $\mathbb{E}^f_{\alpha,\beta}(\xi_1)\le\mathbb{E}^f_{\alpha,\beta}(\xi_2)$.
\item[(iii)] Let $\xi\in L^p(\mathcal{F}_{\beta})$. For every $A\in\mathcal{F}_{\alpha}$,
\[
\mathbf{1}_A\mathbb{E}^f_{\alpha,\beta}(\xi)=\mathbb{E}^{f_A}_{\alpha,\beta}(\mathbf{1}_A\xi),
\]
where $f_A(t,y,z)=f(t,y,z)\mathbf{1}_A\mathbf{1}_{\{t\ge\alpha\}}$.
\item[(iv)] Let $\xi\in L^p(\mathcal{F}_{\beta})$. For every $\gamma\in\mathcal{T}$ such that $\gamma\ge\beta$,
\[
\mathbb{E}^f_{\alpha,\beta}(\xi)=\mathbb{E}^{f^{\beta}}_{\alpha,\gamma}(\xi),
\]
where $f^{\beta}(t,y,z)=f(t,y,z)\mathbf{1}_{\{t\le\beta\}}$.

\item[(v)] Let $p=1$. Assume that $f_1,f_2$ satisfies \textnormal{(H1)-(H5),(Z)} 
and let $\alpha,\beta_1,\beta_2\in\mathcal{T}$, $\alpha\le\beta_1\le\beta_2$. 
Assume that $\xi_1\in L^1(\mathcal{F}_{\beta_1})$ and $\xi_2\in L^1(\mathcal{F}_{\beta_2})$. 
Moreover, let $(Y^1,Z^1)$ be a solution to \textnormal{BSDE}$^{\alpha,\beta_2}(\xi_2,f_1^{\beta_1})$, 
where $f_1^{\beta_1}(t,y,z)=f_1(t,y,z)\mathbf{1}_{\{t\le\beta\}}$, and $(Y^2,Z^2)$ 
be a solution to \textnormal{BSDE}$^{\alpha,\beta_2}(\xi_2,f_2)$,
such that $Y^1, Y^2$ are of class (D). If $(Y^1-Y^2)\in\mathcal{S}^2_{\mathbb{F}}(\alpha,\beta_2)$, then
\begin{equation*}
\begin{split}
&|\mathbb{E}^{f_1}_{\alpha,\beta_1}(\xi_1)-\mathbb{E}^{f_2}_{\alpha,\beta_2}(\xi_2)|^2\le C \mathbb{E}\Big(\int^{\beta_1}_\alpha|Y^1_r-Y^2_r||f_1-f_2|(r,Y^2_r,Z^2_r)\,dr\\
&\quad+|\xi_1-\xi_2|^2+\int^{\beta_2}_{\beta_1}|Y^1_r-Y^2_r||f_2(r,Y^2_r,Z^2_r)|\,dr|\mathcal{F}_{\alpha}\Big),
\end{split}
\end{equation*}
for some $C$ depending only on $\lambda,\mu, T$.
\item[(vi)] Let $p>1$. Assume that $f_1,f_2$ satisfies \textnormal{(H1)-(H5)} and let 
$\alpha,\beta_1,\beta_2\in\mathcal{T}$, $\alpha\le\beta_1\le\beta_2$. 
Assume that $\xi_1\in L^p(\mathcal{F}_{\beta_1})$ and $\xi_2\in L^p(\mathcal{F}_{\beta_2})$. 
Moreover, let $(Y^1,Z^1)$ be a solution to \textnormal{BSDE}$^{\beta_2}(\xi_1,f_1^{\beta_1})$, with $Y^1\in \mathcal S^p_{\mathbb F}(0,\beta_2)$,
where $f_1^{\beta_1}(t,y,z)=f_1(t,y,z)\mathbf{1}_{\{t\le\beta_1\}}$, and $(Y^2,Z^2)$ 
be a solution to \textnormal{BSDE}$^{\beta_2}(\xi_2,f_2)$, with  $Y^2\in \mathcal S^p_{\mathbb F}(0,\beta_2)$. Then there exists $c>0$, depending only on $T,\mu,\lambda,p$, such that 
such that 
\begin{equation*}
\begin{split}
&\|\mathbb{E}^{f_1}_{\cdot,\beta_1}(\xi_1)-\mathbb{E}^{f_2}_{\cdot,\beta_2}(\xi_2)\|_{\mathcal S^p(0,\beta_1)}\le c \Big[ \mathbb E\Big(\int^{\beta_1}_0|f_1-f_2|(r,Y^2_r,Z^2_r)\,dr\Big)^p\\
&\quad+\mathbb E|\xi_1-\xi_2|^p+\mathbb E\Big(\int^{\beta_2}_{\beta_1}|f_2(r,Y^2_r,Z^2_r)|\,dr\Big)^p\Big]^{1/p}.
\end{split}
\end{equation*}
\end{enumerate}
\end{proposition}

\begin{proof}
(i) Assume that $V$ is increasing. Let $\sigma,\tau\in \mathcal{T}$ be such that 
$\alpha\le\sigma\le\tau\le\beta$. Obviously, $(X,H)$ is a solution to BSDE$^{\alpha,\tau}(X_{\tau},f+dV)$. 
Let $(\tilde{X},\tilde{H})$ be a solution to BSDE$^{\alpha,\tau}(X_{\tau},f)$  such that $\tilde X$
is of class \textnormal{(D)}, in case $p=1$, and $\tilde X\in \mathcal S^p_{\mathbb F}(\alpha,\tau)$, in case $p>1$. 
By \cite[Proposition 3.2, Lemma 3.3]{KRzS2} and  Theorem \ref{12sierpnia1}, $X\ge\tilde{X}$ on $[[\alpha,\tau]]$, 
in particular, $X_{\sigma}\ge\tilde{X}_{\sigma}$. Therefore, we have  
$\mathbb{E}^f_{\sigma,\tau}(X_{\tau})=\tilde{X}_{\sigma}\le X_{\sigma}$, 
hence $X$ is $\mathbb{E}^f$-supermartingale. Analogously,  we show that if $-V$ is 
increasing, then $X$ is $\mathbb{E}^f$-submartingale. This completes the proof of (i).
The assertion (ii)  follows directly from \cite[Proposition 3.2, Lemma 3.3]{KRzS2}
and Theorem \ref{12sierpnia1}.
As to (iii), let $A\in\mathcal{F}_{\alpha}$. Let $(Y,Z)$, $(\bar{Y},\bar{Z})$ be solutions to 
BSDE$^{\beta}(\xi,f)$ and BSDE$^{\beta}(\mathbf{1}_A\xi,f_A)$, respectively, such that 
$Y_{\tau}=\mathbb{E}^f_{\tau,\beta}(\xi)$, $\tau\in \TT^\beta$
and $\bar{Y}_{\tau}=\mathbb{E}^{f_A}_{\tau,\beta}(\mathbf{1}_A\xi)$, $\tau\in\TT^\beta$. 
It is easy to see that $(\mathbf{1}_AY,\mathbf{1}_AZ)$ is a solution to 
BSDE$^{\alpha,\beta}(\mathbf{1}_A\xi,f_A)$. Indeed, for $\sigma\in\mathcal{T}_{\alpha,\beta}$,
\begin{equation*}
\begin{split}
\mathbf{1}_A Y_t&=\mathbf{1}_A\xi+\int^{\beta}_{\sigma}\mathbf{1}_Af(r,Y_r,Z_r)\,dr-\int^{\beta}_{\sigma}\mathbf{1}_AZ_r\,dB_r\\
&=\mathbf{1}_A\xi+\int^{\beta}_{\sigma}\mathbf{1}_Af(r,\mathbf{1}_AY_r,\mathbf{1}_AZ_r)\,dr-\int^{\beta}_{\sigma}\mathbf{1}_AZ_r\,dB_r.
\end{split}
\end{equation*}
By the uniqueness  for BSDEs (see Theorems \ref{5grudnia3},\ref{12sierpnia1}) $\mathbf{1}_AY=\bar{Y}$ on $[[\alpha,\beta]]$, which implies (iii).
For (iv), let $(Y,Z)$ be a solution to BSDE$^{\beta}(\xi,f)$ and let $(\bar{Y},\bar{Z})$ be a solution to BSDE$^{\gamma}(\xi,f^{\beta})$ 
such that  $Y_{\tau}=\mathbb{E}^f_{\tau,\beta}(\xi)$, $\tau\in\TT^\beta$ and $\bar{Y}_{\tau}=\mathbb{E}^{f^{\beta}}_{\tau,\gamma}(\xi)$,
$\tau\in\TT^\gamma$. 
Note that, $(\bar{Y},\bar{Z})$ is a solution to BSDE$^{\alpha,\beta}(\bar{Y}_{\beta},f)$. What is more, for $\sigma\in\mathcal{T}_{\alpha,\beta}$,
\[
\bar{Y}_{\sigma}=\mathbb{E}\Big(\xi+\int^{\gamma}_{\sigma} f^{\beta}(r,\bar{Y}_r,\bar{Z}_r)\,dr|\mathcal{F}_{\sigma}\Big)=\mathbb{E}\Big(\xi+\int^{\beta}_{\sigma} f(r,\bar{Y}_r,\bar{Z}_r)\,dr|\mathcal{F}_{\sigma}\Big),
\]
therefore $\bar{Y}_{\beta}=\xi$. Hence, $(\bar{Y},\bar{Z})$ is a solution to BSDE$^{\alpha,\beta}(\xi,f)$, which results,  by the uniqueness argument, that $\bar{Y}=Y$ on $[[\alpha,\beta]]$ and $\mathbb{E}^f_{\alpha,\beta}(\xi)=\bar{Y}_{\alpha}=\mathbb{E}^{f^{\beta}}_{\alpha,\gamma}$, $\alpha\le\beta$. This concludes the proof of (iv).  Now, we shall proceed to the proof of (v).
Let $(Y^1,Z^1)$, $(Y^2,Z^2)$ be defined as in the assertion (v).
By (iv), we know that $\mathbb{E}^{f_1}_{\alpha,\beta_1}(\xi_1)=Y^1_{\alpha}$, for $\alpha\in\mathcal{T}^{\beta_1}$. Let us define
\[
\tau_k=\inf\Big\{t\ge 0:\,\int^t_0|Z^1_r-Z^2_r|^2\,dr\ge k\Big\}\wedge \beta_2,\quad k\in\mathbb{N}.
\]
From   the definition of a  solution to BSDE, we have that $\{\tau_k\}_{k\ge 1}$ is a chain. 
By Ito's formula, for $\alpha\in\mathcal{T}^{\beta_1}$,
\begin{equation}\label{4grudnia3}
\begin{split}
&e^{a \alpha}|Y^1_{\alpha}-Y^2_{\alpha}|^2+\int^{\tau_k}_{\alpha} e^{a r}|Z^1_r-Z^2_r|^2\,dr+a\int^{\tau_k}_{\alpha} e^{a r}|Y^1_r-Y^2_r|^2\,dr\\
&\quad\le e^{a\tau_k}|Y^1_{\tau_k}-Y^2_{\tau_k}|^2+2\int^{\tau_k}_{\alpha} e^{a r}(Y^1_r-Y^2_r)(f^{\beta_1}_1(r,Y^1_r,Z^1_r)-f_2(r,Y^2_r,Z^2_r))\,dr\\
&\quad-2\int^{\tau_k}_{\alpha} e^{a r}(Y^1_r-Y^2_r)(Z^1_r-Z^2_r)\,dB_r,\quad a\ge 0.
\end{split}
\end{equation}
By (H1), (H2) (without loss of generality we may assume that $\mu=0$), we have 
\begin{align*}
&(Y^1_r-Y^2_r)(f(r,Y^1_r,Z^1_r)-f(r,Y^2_r,Z^2_r))\le \lambda|Y^1_r-Y^2_r||Z^1_r-Z^2_r|\\
&\quad+|Y^1_r-Y^2_r||f^{\beta_1}_1-f_2|(r,Y^2_r,Z^2_r)\le 4\lambda^2|Y^1_r-Y^2_r|^2+\frac{1}{4}|Z^1_r-Z^2_r|^2\\
&\quad+|Y^1_r-Y^2_r||f^{\beta_1}_1-f_2|(r,Y^2_r,Z^2_r).
\end{align*}
Therefore,  by \eqref{4grudnia3}, we have 
\begin{equation*}
\begin{split}
&e^{a \alpha}|Y^1_{\alpha}-Y^2_{\alpha}|^2+\int^{\tau_k}_{\alpha} e^{a r}|Z^1_r-Z^2_r|^2\,dr+a\int^{\tau_k}_{\alpha} e^{a r}|Y^1_r-Y^2_r|^2\,dr\\
&\quad\le e^{a \tau_k}|Y^1_{\tau_k}-Y^2_{\tau_k}|^2+8\lambda^2\int^{\beta_2}_{\alpha}e^{a r}|Y^1_r-Y^2_r|^2\,dr+\frac{1}{2}\int^{\beta_2}_{\alpha}e^{a r}|Z^1_r-Z^2_r|^2\,dr\\
&\quad+2\int^{\beta_2}_{\alpha}e^{a r}|Y^1_r-Y^2_r||f^{\beta_1}_1-f_2|(r,Y^2_r,Z^2_r)\,dr\\
&\quad-2\int^{\tau_k}_{\alpha} e^{a r}(Y^1_r-Y^2_r)(Z^1_r-Z^2_r)\,dB_r.
\end{split}
\end{equation*}
Consequently, by the fact that $\int^{\cdot}_0e^{\alpha r}(Y^1_r-Y^2_r)(Z^1_r-Z^2_r)\,dB_r$ is a martingale on $[[\alpha,\tau_k]]$,
we find that for $a\ge 8\lambda^2$,
\begin{equation}\label{4grudnia4}
\begin{split}
e^{a \alpha}|Y^1_{\alpha}-Y^2_{\alpha}|^2&\le \mathbb{E}\Big(e^{a \tau_k}|Y^1_{\tau_k}-Y^2_{\tau_k}|^2\\
&\quad+2\int^{\beta_2}_{\alpha}e^{a r}|Y^1_r-Y^2_r||f^{\beta_1}_1-f_2|(r,Y^2_r,Z^2_r)\,dr|\mathcal{F}_{\alpha}\Big).
\end{split}
\end{equation}
Since $Y^1-Y^2\in\mathcal{S}^2_{\mathbb{F}}(\alpha,\beta_2)$, 
we may conclude, by letting  $k\to\infty$ in the right-hand side of \eqref{4grudnia4}  and applying the Lebesgue dominated convergence theorem, that
\begin{equation*}
\begin{split}
e^{a \alpha}|Y^1_{\alpha}-Y^2_{\alpha}|^2&\le \mathbb{E}\Big(e^{a \beta_2}|\xi_1-\xi_2|^2\\
&\quad+2\int^{\beta_2}_{\alpha}e^{a r}|Y^1_r-Y^2_r||f^{\beta_1}_1-f_2|(r,Y^2_r,Z^2_r)\,dr|\mathcal{F}_{\alpha}\Big),
\end{split}
\end{equation*}
which implies that
\begin{equation}\label{4grudnia5}
\begin{split}
|Y^1_{\alpha}-Y^2_{\alpha}|^2&\le C \mathbb{E}\Big(|\xi_1-\xi_2|^2+\int^{\beta_2}_{\alpha}|Y^1_r-Y^2_r||f^{\beta_1}_1-f_2|(r,Y^2_r,Z^2_r)\,dr|\mathcal{F}_{\alpha}\Big)
\end{split}
\end{equation}
for some $C>0$ depending only on $\lambda$ and $\beta_2$. Finally, note that
\begin{equation*}
\begin{split}
\int^{\beta_2}_{\alpha}|Y^1_r-Y^2_r||f^{\beta_1}_1-f_2|(r,Y^2_r,Z^2_r)\,dr&=\int^{\beta_1}_{\alpha}|Y^1_r-Y^2_r||f^{\beta_1}_1-f_2|(r,Y^2_r,Z^2_r)\,dr\\
&\quad+\int^{\beta_2}_{\beta_1}|Y^1_r-Y^2_r||f_2(r,Y^2_r,Z^2_r)|\,dr,
\end{split}
\end{equation*}
which combined with \eqref{4grudnia5} completes the proof of (v). The inequality asserted in (vi) follows directly from Proposition \ref{13stycznia19}.
\end{proof}

\section{Extended nonlinear Dynkin games}\label{roz6}
\label{sec9}

In the whole section, we assume that (H1)--(H5), (Z) are in force and that $L$ and $U$ are $\mathbb{F}$-optional processes of class (D).

\begin{definition}
Let $\tau\in\mathcal{T}$ and $H\in\mathcal{F}_{\tau}$. A pair $\rho=(\tau,H)$ is called a {\em stopping system} if $\{\tau=T\}\subset H$.
For brevity, we write $\tau\lfloor H$.
\end{definition}

By $\mathcal{U}$ we denote the set of all stoping systems.  We then have $\mathcal{T}\subset\mathcal U$, by using embedding  $\mathcal{T}\ni \tau\mapsto \tau\lfloor\Omega\in \mathcal U$.
For given $\theta\in\mathcal{T}$, we denote  by $\mathcal{U}_{\theta}$ the set of all stopping systems $\tau\lfloor H$ such that $\tau \ge\theta$.
For an optional, right-limited process $\phi$ and  $\tau\lfloor H\in\mathcal U$ we put
\[
\phi_{\tau\lfloor H}:=\phi_{\tau}\mathbf{1}_{H}+\phi_{\tau+}\mathbf{1}_{H^c}.
\]
In particular, we have  $\phi_{\tau\lfloor \Omega}=\phi_{\tau}$.
For an optional process $\phi$ we let
\[
 \phi^u_{\tau\lfloor H}:=\phi_{\tau}\mathbf{1}_{H}+\overleftarrow{\phi}_{\tau}\mathbf{1}_{H^c}\quad\mathrm{and}\quad\phi^l_{\tau\lfloor H}:=\phi_{\tau}\mathbf{1}_{H}+\underleftarrow{\phi}_{\tau}\mathbf{1}_{H^c}.
\]
Note that, when $\phi$ is right-limited, then $\phi^u_{\tau\lfloor H}=\phi^l_{\tau\lfloor H}=\phi_{\tau\lfloor H}$.

For two stopping systems $\tau\lfloor H,\sigma\lfloor G\in\mathcal{U}$ we define the pay-off 
\begin{equation}\label{dyn27}
J(\tau\lfloor H,\sigma\lfloor G):=L^u_{\tau\lfloor H}\mathbf{1}_{\{\tau \le\sigma<T\}}+U^l_{\sigma\lfloor G}\mathbf{1}_{\{\sigma<\tau\}}+\xi\mathbf{1}_{\{\tau=\sigma=T\}}.
\end{equation}
Note that $J(\tau\lfloor H,\sigma\lfloor G)$ is $\mathcal{F}_{\tau\wedge\sigma}$-measurable random variable. 
Now, we shall proceed to the so called {\em extended Dynkin games}.

\begin{definition}
Let $\theta\in\mathcal{T}$.
\begin{enumerate}
\item[(i)] Upper and lower value of the game are defined respectively as
\begin{equation}\label{13sierpnia1}
\begin{split}
&\overline{V}(\theta):=\mathop{\mathrm{ess\,inf}}_{\sigma\lfloor G\in\mathcal{U}_{\theta}}\mathop{\mathrm{ess\,sup}}_{\tau\lfloor H\in\mathcal{U}_{\theta}}\mathbb{E}^f_{\theta,\tau\wedge\sigma}J(\tau\lfloor H,\sigma\lfloor G);\\
&\underline{V}(\theta):=\mathop{\mathrm{ess\,sup}}_{\tau\lfloor H\in\mathcal{U}_{\theta}}\mathop{\mathrm{ess\,inf}}_{\sigma\lfloor G\in\mathcal{U}_{\theta}}\mathbb{E}^f_{\theta,\tau\wedge\sigma}J(\tau\lfloor H,\sigma\lfloor G).
\end{split}
\end{equation}
\item[(ii)] We say that an extended $\mathbb{E}^f$-Dynkin game with pay-off function $J$ has a value  if $\overline{V}(\theta)=\underline{V}(\theta)$
for any $\theta\in\mathcal{T}$.
\end{enumerate}
\end{definition}

Let $(Y,Z,R)$ be a solution to RBSDE$^T(\xi,f,L,U)$. For every $\theta\in\mathcal{T}$ and $\varepsilon>0$ we define the following sets
\begin{equation*}
\begin{split}
&A^{\varepsilon}:=\{(\omega,t)\in\Omega\times[0,T]: \,Y_t(\omega)\le L^{\xi}_t(\omega)+\varepsilon\};\\ &B^{\varepsilon}:=\{(\omega,t)\in\Omega\times[0,T]:\,Y_t(\omega)\ge U^{\xi}_t(\omega)-\varepsilon\},
\end{split}
\end{equation*}
where $L^{\xi}_{\theta}:=L_{\theta}\mathbf{1}_{\{\theta<T\}}+\xi\mathbf{1}_{\{\theta=T\}}$, $U^{\xi}_{\theta}:=U_{\theta}\mathbf{1}_{\{\theta<T\}}+\xi\mathbf{1}_{\{\theta=T\}}$ for  $\theta \in \mathcal{T}$.
Let us also define the following stopping times
\begin{equation*}
\tau^{\varepsilon}_{\theta}:=\inf\{t\ge\theta,\,Y_t\le L^{\xi}_t+\varepsilon\}\wedge T,\quad\sigma^{\varepsilon}_{\theta}:=\inf\{t\ge\theta,\,Y_t\ge U^{\xi}_t-\varepsilon\}\wedge T.
\end{equation*} 
We let
\[
H^{\varepsilon}:=\{\omega\in\Omega:\,(\omega,\tau^{\varepsilon}_{\theta}(\omega))\in A^{\varepsilon}\};\quad G^{\varepsilon}:=\{\omega\in\Omega:\,(\omega,\sigma^{\varepsilon}_{\theta}(\omega))\in B^{\varepsilon}\}.
\]
Consider the following stopping systems 
\begin{equation}\label{dyn39}
\tau^{\varepsilon}_{\theta}\lfloor H^{\varepsilon}\quad\mathrm{and}\quad\sigma^{\varepsilon}_{\theta}\lfloor G^{\varepsilon}.
\end{equation}

\begin{lemma}\label{19sierpnia1}
Let $(Y,Z,R)$ be a solution to \textnormal{RBSDE}$^T(\xi,f,L,U)$. 
Then $Y$ is an $\mathbb{E}^f$-submartingale on $[[\theta,\tau^{\varepsilon}_{\theta}]]$ and an $\mathbb{E}^f$-supermartingale on $[[\theta,\sigma^{\varepsilon}_{\theta}]]$. 
\end{lemma}
\begin{proof}
We show that $Y$ is $\mathbb{E}^f$-submartingale on $[[\theta,\tau^{\varepsilon}_{\theta}]]$. The proof of the second assertion runs analogously. By the definition of $\tau^{\varepsilon}_{\theta}$ we have  $Y_t>L_t+\varepsilon$ on $[[\theta,\tau^{\varepsilon}_{\theta}[[$. This implies,  by the minimality condition, that $R^+$ is constant on $[[\theta,\tau^{\varepsilon}_{\theta}[[$. By  the last inequality,  we also have  $Y_{\tau^{\varepsilon}_{\theta}-}\ge\overrightarrow{L}_{\tau^{\varepsilon}_{\theta}}+\varepsilon$, so as a result,  by the minimality condition again, we get $\Delta^-R^+_{\tau^{\varepsilon}_{\theta}}=0$. Therefore,  $R^+$ is constant on $[[\theta,\tau^{\varepsilon}_{\theta}]]$. This implies that
\[
Y_t=Y_{\tau^{\varepsilon}_{\theta}}+\int^{\tau^{\varepsilon}_{\theta}}_tf(r,Y_r,Z_r)\,dr-R^-_{\tau^{\varepsilon}_{\theta}}+R^-_t-\int^{\tau^{\varepsilon}_{\theta}}_t Z_r\,dB_r,\quad t\in [\theta,\tau^\varepsilon_\theta].
\]
Thus, by Proposition \ref{nonlinprop} (i),  $Y$ is an $\mathbb{E}^f$-submartingale on $[[\theta,\tau^{\varepsilon}_{\theta}]]$. 
\end{proof}

\begin{lemma}\label{13sierpnia2}
Let $(Y,Z,R)$ be a solution to \textnormal{RBSDE}$^T(\xi,f,L,U)$. 
\begin{enumerate}
\item[(i)] For any $\theta\in\mathcal{T}$,
\begin{equation}\label{13sierpnia3}
Y_{\tau^{\varepsilon}_{\theta}\lfloor H^{\varepsilon}}\le L^{\xi,u}_{\tau^{\varepsilon}_{\theta}\lfloor H^{\varepsilon}}+\varepsilon\quad\mathrm{and}\quad Y_{\sigma^{\varepsilon}_{\theta}\lfloor G^{\varepsilon}}\ge U^{\xi,l}_{\sigma^{\varepsilon}_{\theta}\lfloor G^{\varepsilon}}-\varepsilon.
\end{equation}
\item[(ii)] For each  $\theta\in\mathcal{T}$ and any  $\tau\lfloor A,\sigma\lfloor B\in\mathcal{U}_\theta$,
\begin{equation}\label{13sierpnia4}
\begin{split}
&\mathbb{E}^f_{\theta,\tau^{\varepsilon}_{\theta}\wedge\sigma}(Y_{\tau^{\varepsilon}_{\theta}\lfloor H^{\varepsilon}}\mathbf{1}_{\{\tau^{\varepsilon}_{\theta}\le\sigma\}}
+Y_{\sigma\lfloor B}\mathbf{1}_{\{\sigma<\tau^{\varepsilon}_{\theta}\}})\ge Y_{\theta}
\quad\mathrm{and}\\
&\mathbb{E}^f_{\theta,\tau\wedge\sigma^{\varepsilon}_{\theta}}(Y_{\tau\lfloor A}\mathbf{1}_{\{\tau\le\sigma^{\varepsilon}_{\theta}\}}
+Y_{\sigma^{\varepsilon}_{\theta}\lfloor G^{\varepsilon}}\mathbf{1}_{\{\sigma^{\varepsilon}_{\theta}<\tau\}})\le Y_{\theta}.
\end{split}
\end{equation}
\end{enumerate}
\end{lemma}
\begin{proof}
 \textit{(i)} We shall prove the first inequality in \eqref{13sierpnia3}, the proof of the second one runs analogously. Due to the definitions of $\tau^{\varepsilon}_{\theta}\lfloor H^{\varepsilon},Y_{\tau^{\varepsilon}_{\theta}\lfloor H^{\varepsilon}},L^{\xi,u}_{\tau^{\varepsilon}_{\theta}\lfloor H^{\varepsilon}}$ and $H^{\varepsilon}$,  we have, on the set $H^{\varepsilon}$, that   $Y_{\tau^{\varepsilon}_{\theta}\lfloor H^{\varepsilon}}=Y_{\tau^{\varepsilon}_{\theta}}\le L^{\xi}_{\tau^{\varepsilon}_{\theta}}+\varepsilon=L^{\xi,u}_{\tau^{\varepsilon}_{\theta}\lfloor H^{\varepsilon}}+\varepsilon$,
while on the set $H^{\varepsilon,c}$, we have
\begin{equation}\label{13sierpnia5}
Y_{{\tau^{\varepsilon}_{\theta}\lfloor H^{\varepsilon}}}=Y_{\tau^{\varepsilon}_{\theta}+}\quad\mathrm{and}\quad L^{\xi,u}_{{\tau^{\varepsilon}_{\theta}\lfloor H^{\varepsilon}}}=\overleftarrow{L}^{\xi}_{\tau^{\varepsilon}_{\theta}}.
\end{equation}
On the other hand, by the definition of $\tau^{\varepsilon}_{\theta}$, 
for  P-a.e. $\omega\in\Omega$
 there exists nonincreasing sequence $(t_n)$ (depending on $\omega\in\Omega$) such that 
 $t_n\searrow\tau^{\varepsilon}_{\theta}$ and $Y_{t_n}\le L^{\xi}_{t_n}+\varepsilon$
 for all $n\in\mathbb{N}$. Therefore
 \[
 \limsup_{n\rightarrow\infty}Y_{t_n}\le\limsup_{n\rightarrow\infty}L^{\xi}_{t_n}+\varepsilon.
 \] 
Due to  the definiton of $\overleftarrow{L}^{\xi}$,  we have  
$\limsup_{n\rightarrow\infty}L^{\xi}_{t_n}\le\overleftarrow{L}^{\xi}_{\tau^{\varepsilon}_{\theta}}$. 
Since $Y$ is regulated,  $\limsup_{n\rightarrow\infty}Y_{t_n}=Y_{\tau^{\varepsilon}_{\theta}+}$. 
Thus, $Y_{\tau^{\varepsilon}_{\theta}+}\le\overleftarrow{L}^{\xi}_{\tau^{\varepsilon}_{\theta}}+\varepsilon$. 
This inequality combined with \eqref{13sierpnia5} implies that $Y_{\tau^{\varepsilon}_{\theta}\lfloor H^{\varepsilon}}\le L^{\xi,u}_{\tau^{\varepsilon}_{\theta}\lfloor H^{\varepsilon}}+\varepsilon$ 
on $H^{\varepsilon,c}$.

 \textit{(ii)} First, we shall prove the first inequality in \eqref{13sierpnia4}. We have  
\begin{equation}
\label{eq.fr0}
\begin{split}
&Y_{\tau^{\varepsilon}_{\theta}\lfloor H^{\varepsilon}}\mathbf{1}_{\{\tau^{\varepsilon}_{\theta}\le\sigma\}}
+Y_{\sigma\lfloor B}\mathbf{1}_{\{\sigma<\tau^{\varepsilon}_{\theta}\}}
=Y_{\tau^{\varepsilon}_{\theta}}\mathbf{1}_{H^{\varepsilon}}\mathbf{1}_{\{\tau^{\varepsilon}_{\theta}\le\sigma\}}
+Y_{\tau^{\varepsilon}_{\theta}+}\mathbf{1}_{H^{\varepsilon,c}}\mathbf{1}_{\{\tau^{\varepsilon}_{\theta}\le\sigma\}}
\\
&\quad+Y_{\sigma}\mathbf1_{B}\mathbf{1}_{\{\sigma<\tau^{\varepsilon}_{\theta}\}}
+Y_{\sigma+}\mathbf1_{B^c}\mathbf{1}_{\{\sigma<\tau^{\varepsilon}_{\theta}\}}.
\end{split}
\end{equation}
By Lemma \ref{19sierpnia1}  $Y$ is an $\mathbb{E}^f$-submartingale on $[[\theta,\tau^{\varepsilon}_{\theta}]]$, which implies that 
\begin{equation}
\label{eq.fr1}
\mathbf{1}_{\{\sigma<\tau^{\varepsilon}_{\theta}\}}Y_{\sigma+}\ge \mathbf{1}_{\{\sigma<\tau^{\varepsilon}_{\theta}\}}Y_{\sigma\wedge \tau^{\varepsilon}_{\theta}}.
\end{equation}
By the form of $H^\varepsilon$ and the {\em minimality condition}, $\mathbf1_{H^{\varepsilon,c}}\Delta^+R^+_{\tau^{\varepsilon}_{\theta}}=0$.
Thus,
\begin{equation}
\label{eq.fr2}
\mathbf{1}_{H^{\varepsilon,c}}Y_{\tau^{\varepsilon}_{\theta}+}\ge \mathbf{1}_{H^{\varepsilon,c}}Y_{\tau^{\varepsilon}_{\theta}}.
\end{equation}
By virtue of \eqref{eq.fr0}--\eqref{eq.fr2}, we conclude that   $Y_{\tau^{\varepsilon}_{\theta}\lfloor H^{\varepsilon}}\mathbf{1}_{\{\tau^{\varepsilon}_{\theta}\le\sigma\}}+Y_{\sigma\lfloor B}\mathbf{1}_{\{\sigma<\tau^{\varepsilon}_{\theta}\}}\ge Y_{\tau^{\varepsilon}_{\theta}\wedge\sigma}$. 
Since the operator $\mathbb{E}^f$ is nondecreasing (see Proposition \ref{nonlinprop} (ii))
\[
\mathbb{E}^f_{\theta,\tau^{\varepsilon}_{\theta}\wedge\sigma}(Y_{\tau^{\varepsilon}_{\theta}\lfloor H^{\varepsilon}}\mathbf{1}_{\{\tau^{\varepsilon}_{\theta}\le\sigma\}}+Y_{\sigma\lfloor B}\mathbf{1}_{\{\sigma<\tau^{\varepsilon}_{\theta}\}})\ge\mathbb{E}^f_{\theta,\tau^{\varepsilon}_{\theta}\wedge\sigma}(Y_{\tau^{\varepsilon}_{\theta}\wedge\sigma}).
\]
Using again the fact that $Y$ is $\mathbb{E}^f$-submartingale on $[[\theta,\tau^{\varepsilon}_{\theta}]]$, we conclude that 
$\mathbb{E}^f_{\theta,\tau^{\varepsilon}_{\theta}\wedge\sigma}(Y_{\tau^{\varepsilon}_{\theta}\wedge\sigma})\ge Y_{\theta}$, and as a result 
$\mathbb{E}^f_{\theta,\tau^{\varepsilon}_{\theta}\wedge\sigma}(Y_{\tau^{\varepsilon}_{\theta}\lfloor H^{\varepsilon}}\mathbf{1}_{\{\tau^{\varepsilon}_{\theta}\le\sigma\}}
+Y_{\sigma\lfloor B}\mathbf{1}_{\{\sigma<\tau^{\varepsilon}_{\theta}\}})\ge Y_{\theta}$. 

The proof of the second inequality in \eqref{13sierpnia4} requires slightly different arguments. We have
\begin{equation}
\label{eq.fr3}
\begin{split}
&Y_{\tau\lfloor A}\mathbf{1}_{\{\tau\le\sigma^{\varepsilon}_{\theta}\}}+Y_{\sigma^{\varepsilon}_{\theta}\lfloor G^{\varepsilon}}\mathbf{1}_{\{\sigma^{\varepsilon}_{\theta}<\tau\}}
=Y_{\tau}\mathbf{1}_{A}\mathbf{1}_{\{\tau\le\sigma^{\varepsilon}_{\theta}\}}
+Y_{\tau+}\mathbf{1}_{A^c}\mathbf{1}_{\{\tau\le\sigma^{\varepsilon}_{\theta}\}}
\\
&\quad+Y_{\sigma^{\varepsilon}_{\theta}}\mathbf1_{G^\varepsilon}\mathbf{1}_{\{\sigma^{\varepsilon}_{\theta}<\tau\}}
+Y_{\sigma^{\varepsilon}_{\theta}+}\mathbf1_{G^{\varepsilon,c}}\mathbf{1}_{\{\sigma^{\varepsilon}_{\theta}<\tau\}}.
\end{split}
\end{equation}
By the analogous argument as in the proof of \eqref{eq.fr0} -  the form of $G^\varepsilon$ combined with the definition of a solution to RBSDE give $\mathbf1_{G^{\varepsilon,c}}\Delta^+R^-_{\sigma^{\varepsilon}_{\theta}}=0$ - we find that 
\begin{equation}
\label{eq.fr4}
\mathbf1_{G^{\varepsilon,c}}Y_{\sigma^{\varepsilon}_{\theta}+}\le \mathbf1_{G^{\varepsilon,c}}Y_{\sigma^{\varepsilon}_{\theta}}.
\end{equation}
Using the fact that $Y$ is an $\mathbb{E}^f$-supermartingale on $[[\theta,\sigma^{\varepsilon}_{\theta}]]$ (see Lemma \ref{19sierpnia1}) we obtain that 
\begin{equation}
\label{eq.fr5}
\mathbf{1}_{\{\tau<\sigma^{\varepsilon}_{\theta}\}}Y_{\tau+}\le \mathbf{1}_{\{\tau<\sigma^{\varepsilon}_{\theta}\}}Y_{\tau}.
\end{equation}
But we also need 
\begin{equation}
\label{eq.fr6}
\mathbf{1}_{\{\tau=\sigma^{\varepsilon}_{\theta}\}}Y_{\tau+}\le \mathbf{1}_{\{\tau=\sigma^{\varepsilon}_{\theta}\}}Y_{\tau}.
\end{equation}
The above inequality may not hold only in case $\Delta^+R^-_{\sigma^{\varepsilon}_{\theta}}>0$.
The last relation implies, by using the {\em minimality condition}, that $Y_{\sigma^{\varepsilon}_{\theta}}=U_{\sigma^{\varepsilon}_{\theta}}$.
On the other hand, by the definition of $\sigma^{\varepsilon}_{\theta}$, $U_t-Y_t\ge\varepsilon,\, t\in [\theta,\sigma^{\varepsilon}_{\theta})$.
This and the previous equation force left positive jump, so $\Delta R^-_{\sigma^{\varepsilon}_{\theta}}>0$. Consequently, by the {\em minimality condition}, $Y_{\sigma^{\varepsilon}_{\theta}-}=\underleftarrow{U}_{\sigma^{\varepsilon}_{\theta}}$. This contradicts the relation 
$U_t-Y_t\ge\varepsilon,\, t\in [\theta,\sigma^{\varepsilon}_{\theta})$. Therefore, \eqref{eq.fr6} must hold. Combining \eqref{eq.fr3}--\eqref{eq.fr6}
gives 
\[
Y_{\tau\lfloor A}\mathbf{1}_{\{\tau\le\sigma^{\varepsilon}_{\theta}\}}+Y_{\sigma^{\varepsilon}_{\theta}\lfloor G^{\varepsilon}}\mathbf{1}_{\{\sigma^{\varepsilon}_{\theta}<\tau\}}
\le Y_{\tau\wedge \sigma^{\varepsilon}_{\theta}}.
\] 
With the aid of  monotonicity of the operator  $\mathbb{E}^f_{\theta,\tau\wedge\sigma^{\varepsilon}_{\theta}}$ and the fact that $Y$ is 
$\mathbb{E}^f$-supermartingale on $[[\theta,\sigma^{\varepsilon}_{\theta}]]$, we easily deduce from the above inequality the result.
\end{proof}

\begin{lemma}\label{16sierpnia2}
Let $(Y,Z,R)$ be a solution to \textnormal{RBSDE}$^T(\xi,f,L,U)$. We have the following inequalities:
\begin{equation}\label{16sierpnia3}
\mathbb{E}^f_{\theta,\tau\wedge\sigma^{\varepsilon}_{\theta}}J(\tau\lfloor A,\sigma^{\varepsilon}_{\theta}\lfloor G^{\varepsilon})-C\varepsilon\le Y_{\theta}\le \mathbb{E}^f_{\theta,\tau^{\varepsilon}_{\theta}\wedge\sigma}J(\tau^{\varepsilon}_{\theta}\lfloor H^{\varepsilon},\sigma\lfloor B)+C\varepsilon,
\end{equation}
\end{lemma}
where $C$ is a constant depending only on $\lambda,\mu,\kappa,T, \|g\|_{L^1},\gamma$.
\begin{proof}
Let $\theta\in\mathcal{T}$ and $\varepsilon>0$. We shall show  the first  inequality  in \eqref{16sierpnia3} (the proof of the other one is analogous).
By Lemma \ref{13sierpnia2}
\begin{equation}\label{16sierpnia4}
Y_{\theta}\le\mathbb{E}^f_{\theta,\tau^{\varepsilon}_{\theta}\wedge\sigma}(Y_{\tau^{\varepsilon}_{\theta}\lfloor H^{\varepsilon}}\mathbf{1}_{\{\tau^{\varepsilon}_{\theta}\le\sigma\}}+Y_{\sigma\lfloor B}\mathbf{1}_{\{\sigma<\tau^{\varepsilon}_{\theta}\}}).
\end{equation}
We have 
\[
Y_{\tau^{\varepsilon}_{\theta}\lfloor H^{\varepsilon}}\mathbf{1}_{\{\tau^{\varepsilon}_{\theta}\le\sigma\}}+Y_{\sigma\lfloor B}\mathbf{1}_{\{\sigma<\tau^{\varepsilon}_{\theta}\}}=Y_{\tau^{\varepsilon}_{\theta}\lfloor H^{\varepsilon}}\mathbf{1}_{\{\tau^{\varepsilon}_{\theta}\le\sigma<T\}}+
Y_{\sigma\lfloor B}\mathbf{1}_{\{\sigma<\tau^{\varepsilon}_{\theta}\}}+\xi\mathbf{1}_{\{\tau^{\varepsilon}_{\theta}=\delta=T\}}.
\]
By Lemma \ref{13sierpnia2},  $Y_{\tau^{\varepsilon}_{\theta}\lfloor H^{\varepsilon}}\le L^u_{\tau^{\varepsilon}_{\theta}\lfloor H^{\varepsilon}}+\varepsilon$. Moreover, since $Y\le U$ and $Y$ is right-limited, we have  $Y_{\sigma\lfloor B}=Y^l_{\sigma\lfloor B}\le U^l_{\sigma\lfloor B}$. Consequently, 
\begin{equation*}
\begin{split}
Y_{\tau^{\varepsilon}_{\theta}\lfloor H^{\varepsilon}}\mathbf{1}_{\{\tau^{\varepsilon}_{\theta}\le\sigma\}}+Y_{\sigma\lfloor B}\mathbf{1}_{\{\sigma<\tau^{\varepsilon}_{\theta}\}}&\le(L^u_{\tau^{\varepsilon}_{\theta}\lfloor H^{\varepsilon}}+\varepsilon)\mathbf{1}_{\{\tau^{\varepsilon}_{\theta}\le\sigma<T\}}+U^l_{\sigma\lfloor B}\mathbf{1}_{\{\sigma<\tau^{\varepsilon}_{\theta}\}}\\
&\quad+\xi\mathbf{1}_{\{\tau^{\varepsilon}_{\theta}=\sigma=T\}}\le J(\tau^{\varepsilon}_{\theta}\lfloor H^{\varepsilon},\sigma\lfloor B)+\varepsilon.
\end{split}
\end{equation*}
By \eqref{16sierpnia4} and by properties of the operator $\mathbb{E}^f$ (see Proposition \ref{nonlinprop} (ii) and (v)), we get
\begin{equation}\label{16sierpnia5}
Y_{\theta}\le\mathbb{E}^f_{\theta,\tau^{\varepsilon}_{\theta}\wedge\sigma}(J(\tau^{\varepsilon}_{\theta}\lfloor H^{\varepsilon},\sigma\lfloor B)+\varepsilon)\le\mathbb{E}^f_{\theta,\tau^{\varepsilon}_{\theta}\wedge\sigma}(J(\tau^{\varepsilon}_{\theta}\lfloor H^{\varepsilon},\sigma\lfloor B))+C\varepsilon.
\end{equation}
\end{proof}

\begin{theorem}\label{17sierpnia1}
Let $(Y,Z,R)$ be a solution to \textnormal{RBSDE}$^T(\xi,f,L,U)$. The extended $\mathbb{E}^f$-Dynkin game has a value. What is more, for any stopping time $\theta\in\mathcal{T}$,
\[
\underline{V}(\theta)=Y_{\theta}=\overline{V}(\theta).
\]
Moreover, for every $\theta\in\mathcal{T}$ and $\varepsilon>0$ the pair of stopping systems $(\tau^{\varepsilon}_{\theta}\lfloor H^{\varepsilon},\delta^{\varepsilon}_{\theta}\lfloor G^{\varepsilon})$ defined in \eqref{dyn39} is $\varepsilon$-saddle point in time $\theta$ for extended $\mathbb{E}^f$-Dynkin game, i.e. satisfies inequalities \eqref{13sierpnia3}.
\end{theorem}
\begin{proof}
Since right-hand side inequality in \eqref{16sierpnia4} is satisfied for all $\sigma\lfloor B\in\mathcal{U}_{\theta}$ we have that
\begin{equation*}
\begin{split}
Y_{\theta}&\le\mathop{\mathrm{ess\,inf}}_{\sigma\lfloor B\in\mathcal{U}_{\theta}}\mathbb{E}^f_{\theta,\tau^{\varepsilon}_{\theta}\wedge\sigma}(J(\tau^{\varepsilon}_{\theta}\lfloor H^{\varepsilon},\sigma\lfloor B))+C\varepsilon\\
&\le\mathop{\mathrm{ess\,sup}}_{\tau\lfloor A\in\mathcal{U}_{\theta}}\mathop{\mathrm{ess\,inf}}_{\sigma\lfloor B\in\mathcal{U}_{\theta}}\mathbb{E}^f_{\theta,\tau\wedge\sigma}(J(\tau\lfloor A,\sigma\lfloor B))+C\varepsilon.
\end{split}
\end{equation*}
Thus, by the definition of $\underline{V}_{\theta}$ (see \eqref{13sierpnia1}) we have that $Y_{\theta}\le\underline{V}(\theta)+C\varepsilon$. Similarly, we show that $\overline{V}(\theta)-C\varepsilon\le Y_{\theta}$ for all $\varepsilon>0$. In consequence, $\overline{V}(\theta)\le Y_{\theta}\le\underline{V}(\theta)$, which combined with the obvious inequality $\underline{V}(\theta)\le\overline{V}(\theta)$ gives us $\underline{V}(\theta)=Y_{\theta}=\overline{V}(\theta)$.
\end{proof}

\medskip

\section{Nonlinear Dynkin games}\label{roz7}

Throughout the section, we assume that (H1)--(H5), (Z) are in force and that $L$ and $U$ are $\mathbb{F}$-optional processes of class (D).

For  $\tau,\sigma\in\mathcal T$ we define the pay-off 
\begin{equation}\label{dyn271}
J_0(\tau,\sigma):=L_{\tau}\mathbf{1}_{\{\tau \le\sigma<T\}}+U_{\sigma}\mathbf{1}_{\{\sigma<\tau\}}+\xi\mathbf{1}_{\{\tau=\sigma=T\}}.
\end{equation}

\begin{definition}
Let $\theta\in\mathcal{T}$. Upper and lower value of the game are defined respectively as
\begin{equation}\label{13sierpnia1pr}
\begin{split}
&\overline{V}_0(\theta):=\mathop{\mathrm{ess\,inf}}_{\sigma\ge \theta}\mathop{\mathrm{ess\,sup}}_{\tau\ge \theta}\mathbb{E}^f_{\theta,\tau\wedge\sigma}J_0(\tau,\sigma);\\
&\underline{V}_0(\theta):=\mathop{\mathrm{ess\,sup}}_{\tau\ge \theta}\mathop{\mathrm{ess\,inf}}_{\sigma\ge \theta}\mathbb{E}^f_{\theta,\tau\wedge\sigma}J_0(\tau,\sigma).
\end{split}
\end{equation}
\end{definition}

\begin{lemma}\label{dyn6}
Let $(Y,Z,R)$ be a solution to \textnormal{RBSDE}$^T(\xi,f,L,U)$. 
Assume that $L$ is right upper semicontinuous and $U$ is right lower semicontinuous. 
Then 
\begin{equation}\label{dyn7}
Y_{\tau^{\varepsilon}_{\theta}}\le L^{\xi}_{\tau^{\varepsilon}_{\theta}}+\varepsilon,\quad\quad Y_{\sigma^{\varepsilon}_{\theta}}\ge U^{\xi}_{\sigma^{\varepsilon}_{\theta}}-\varepsilon.
\end{equation}
\end{lemma}
\begin{proof}
In case $\tau^{\varepsilon}_{\theta}=T$,   \eqref{dyn7} is obvious. 
Suppose that  $\tau^{\varepsilon}_{\theta}<T$. 
Suppose by contradiction that  
$P(Y_{\tau^{\varepsilon}_{\theta}}> L^{\xi}_{\tau^{\varepsilon}_{\theta}}+\varepsilon)>0$. 
Without loss of generality we may assume that properties attributed to  $Y,M,R,L,U$ and holding  
$P$-a.s. hold for any  $\omega\in\Omega$. 
Fix  $\omega\in \{Y_{\tau^{\varepsilon}_{\theta}}> L^{\xi}_{\tau^{\varepsilon}_{\theta}}+\varepsilon\}$.
By the minimality condition  $\Delta^+R^+_{\tau^{\varepsilon}_{\theta}}(\omega)=0$, 
and so $\Delta^+Y_{\tau^{\varepsilon}_{\theta}}(\omega)=-(\Delta^+R^+_{\tau^{\varepsilon}_{\theta}}-\Delta^+R^-_{\tau^{\varepsilon}_{\theta}})(\omega)=\Delta^+R^-_{\tau^{\varepsilon}_{\theta}}(\omega)>0$. 
Therefore
\begin{equation}\label{dyn8}
Y_{(\tau^{\varepsilon}_{\theta})+}(\omega)>L^{\xi}_{\tau^{\varepsilon}_{\theta}}(\omega)+\varepsilon.
\end{equation}
Take $\omega\in\Omega$. By the definition of $\tau^{\varepsilon}_{\theta}$ 
there exists a non-increasing sequence  $(t_n(\omega))\searrow\tau^{\varepsilon}_{\theta}(\omega)$ such that  $Y_{t_n}(\omega)\le L^{\xi}_{t_n}(\omega)+\varepsilon$ for any  $n\in\mathbb{N}$. Hence  $\limsup_{n\rightarrow\infty}Y_{t_n}(\omega)\le\limsup_{n\rightarrow\infty}L^{\xi}_{t_n}(\omega)+\varepsilon$. 
By the assumptions made  $L$ is right upper semicontinuous, thus   $\limsup_{n\rightarrow\infty}L^{\xi}_{t_n}(\omega)\le L^{\xi}_{\tau^{\varepsilon}_{\theta}}(\omega)$. 
On the other hand  $t_n(\omega)\searrow\tau^{\varepsilon}_{\theta}(\omega)$ implies  $\limsup_{n\rightarrow\infty}Y_{t_n}(\omega)=Y_{(\tau^{\varepsilon}_{\theta})+}(\omega)$. Consequently, $Y_{(\tau^{\varepsilon}_{\theta})+}(\omega)\le L^{\xi}_{\tau^{\varepsilon}_{\theta}}(\omega)+\varepsilon$, which contradicts  \eqref{dyn8}. From this we deduce that  $Y_{\tau^{\varepsilon}_{\theta}}\le L^{\xi}_{\tau^{\varepsilon}_{\theta}}+\varepsilon$ for  $P$-a.e. $\omega\in\Omega$. 
\end{proof}

\begin{theorem}\label{dyn20}
Assume that  $L$ is right upper semicontinuous and $U$ is right lower semicontinuous.
Let  $(Y,Z,R)$ be a solution to   \textnormal{RBSDE}$^T(\xi,f,L,U)$. Then for any $\theta\in\mathcal{T}$
\begin{equation}\label{dyn9}
Y_{\theta}=\overline{V}_0(\theta)=\underline{V}_0(\theta).
\end{equation}
Moreover,  for any  $(\tau,\sigma)\in\mathcal{T}_{\theta}\times \mathcal T_\theta$
\begin{equation}\label{dyn10}
\mathbb{E}^f_{\theta,\tau\wedge\sigma^{\varepsilon}_{\theta}}J_0(\tau,\sigma^{\varepsilon}_{\theta})-C\varepsilon\le Y_{\theta}\le\mathbb{E}^f_{\theta,\tau^{\varepsilon}_{\theta}\wedge\sigma}J_0(\tau^{\varepsilon}_{\theta},\sigma)+C\varepsilon,
\end{equation}
where $C$ is a constant depending only on $\lambda,\mu,\alpha,T, \|g\|_{L^1},\gamma$.
\end{theorem}
\begin{proof}
Let  $\theta\in\mathcal{T}$  and  $\varepsilon>0$. We shall prove that  
$(\tau^{\varepsilon}_{\theta},\sigma^{\varepsilon}_{\theta})$ satisfies  \eqref{dyn10}. 
By Lemma  \ref{19sierpnia1}  $Y$ is an  $\mathbb{E}^f$-submartingale  on $[[\theta,\tau^{\varepsilon}_{\theta}]]$.
We thus have
\begin{equation}\label{dyn11}
Y_{\theta}\le\mathbb{E}^f_{\theta,\tau^{\varepsilon}_{\theta}\wedge\sigma}[Y_{\tau^{\varepsilon}_{\theta}\wedge\sigma}].
\end{equation}
By the assumptions made on  $L$ and Lemma  \ref{dyn6}, $Y_{\tau^{\varepsilon}_{\theta}}\le L_{\tau^{\varepsilon}_{\theta}}+\varepsilon$. 
From this and the fact that    $Y\le U$ we have
\begin{equation*}
Y_{\tau^{\varepsilon}_{\theta}\wedge\sigma}\le (L_{\tau^{\varepsilon}_{\theta}}+\varepsilon)\mathbf{1}_{\{\tau^{\varepsilon}_{\theta}\le\sigma<T\}}+U_{\sigma}\mathbf{1}_{\{\sigma<\tau^{\varepsilon}_{\theta}\}}+\xi\mathbf{1}_{\{\tau^{\varepsilon}_{\theta}=\sigma=T\}}\le J_0(\tau^{\varepsilon}_{\theta},\sigma)+\varepsilon.
\end{equation*}
Applying  \eqref{dyn11} and properties of the operator $\mathbb{E}^f$ (see Proposition   \ref{nonlinprop} (ii) and (v)) yields
\begin{equation}\label{dyn12}
Y_{\theta}\le\mathbb{E}^f_{\theta,\tau^{\varepsilon}_{\theta}\wedge\sigma}(J_0(\tau^{\varepsilon}_{\theta},\sigma)+\varepsilon)\le\mathbb{E}^f_{\theta,\tau^{\varepsilon}_{\theta}\wedge\sigma}J_0(\tau^{\varepsilon}_{\theta},\sigma)+C\varepsilon.
\end{equation}
By Lemma \ref{19sierpnia1}  $Y$ is an $\mathcal{E}^f$-supermartingale  on  $[[\theta,\sigma^{\varepsilon}_{\theta}]]$. As a result
\begin{equation}\label{dyn13}
Y_{\theta}\ge \mathbb{E}^f_{\theta,\tau\wedge\sigma^{\varepsilon}_{\theta}}(Y_{\tau\wedge\sigma^{\varepsilon}_{\theta}}).
\end{equation}
By the assumptions made on $U$ and Lemma  \ref{dyn6} we have  
$Y_{\sigma^{\varepsilon}_{\theta}}\ge U_{\sigma^{\varepsilon}_{\theta}}-\varepsilon$. 
Applying analogous arguments as in case of $L$ yields
$Y_{\theta}\ge\mathbb{E}^f_{\theta,\tau\wedge\sigma^{\varepsilon}_{\theta}}J_0(\tau,\sigma^{\varepsilon}_{\theta})-C\varepsilon$, 
 which combined with \eqref{dyn12} gives  \eqref{dyn10}. Consequently,
\begin{equation*}
Y_{\theta}\le \mathop{\mathrm{ess\,inf}}_{\sigma\ge \theta}\mathbb{E}^f_{\theta,\tau^{\varepsilon}_{\theta}\wedge\sigma}J_0(\tau^{\varepsilon}_{\theta},\sigma)+\varepsilon\le\mathop{\mathrm{ess\,sup}}_{\tau\ge \theta} \mathop{\mathrm{ess\,inf}}_{\sigma\ge \theta}\mathbb{E}^f_{\theta,\tau\wedge\sigma}J_0(\tau,\sigma)+\varepsilon,
\end{equation*}
which combined with the definition of  $\underline{V}_0(\theta)$ yields  $Y_{\theta}\le\underline{V}(\theta)+\varepsilon$. 
Analogous reasoning gives $\overline{V}(\theta)-\varepsilon\le Y_{\theta}$. Letting  $\varepsilon\to 0$ we find that  $\overline{V}(\theta)\le Y_{\theta}\le\underline{V}(\theta)$, which combined with the obvious inequality $\underline{V}(\theta)\le\overline{V}(\theta)$ gives $\underline{V}(\theta)=Y_{\theta}=\overline{V}_{\theta}$. 

\end{proof}

\section{Existence of saddle points.}\label{roz8}

In the whole section, we assume that (H1)--(H5), (Z) are in force and that $L$ and $U$ are $\mathbb{F}$-optional processes of class (D).

Let $(Y,Z,R)$ be a solution to \textnormal{RBSDE}$^T(\xi,f,L,U)$. 
We shall prove that there exists a saddle point for a nonlinear Dynkin game with sufficiently regular payoffs. 
For  $\theta\in\mathcal{T}$ we define:
\begin{equation}\label{dyn14}
\tau^*_{\theta}:=\inf\{t\ge\theta,\,Y_t=L^{\xi}_t\}\wedge T;\quad\sigma^*_{\theta}:=\inf\{t\ge\theta,\,Y_t=U^{\xi}_t\}\wedge T
\end{equation}
and
\begin{equation}\label{dyn15}
\begin{split}
&\bar{\tau}_{\theta}:=\inf\{t\ge\theta,\,R^{+}_t>R^{+}_{\theta}\}\wedge T;\quad\bar{\sigma}_{\theta}:=\inf\{t\ge\theta,\,R^{-}_t>R^{-}_{\theta}\}\wedge T.
\end{split}
\end{equation}

\begin{proposition}\label{dyn23pr}
Assume that  $L$ is right upper semicontinuous and $U$ is right lower semicontinuous.
Let  $(Y,Z,R)$ be a solution to   \textnormal{RBSDE}$^T(\xi,f,L,U)$ and let $\theta\in\mathcal T$.
\begin{enumerate}
\item [1)] If $R^{-,*}$ is continuous, then   $Y$ is an  $\mathbb{E}^f$-supermartingale on  $[[\theta,\bar{\sigma}_{\theta}]]$. Moreover,
\begin{equation}\label{dyn16}
Y_{\sigma^*_{\theta}}=U^{\xi}_{\sigma^*_{\theta}}\quad\mathrm{and}\quad Y_{\bar{\sigma}_{\theta}}=U^{\xi}_{\bar{\sigma}_{\theta}}.
\end{equation}
\item[2)] If  $R^{+,*}$ is continuous, then  $Y$ is an  $\mathbb{E}^f$-submartingale on   $[[\theta,\bar{\tau}_{\theta}]]$. Moreover, 
\begin{equation}\label{dyn17}
Y_{\tau^*_{\theta}}=L^{\xi}_{\tau^*_{\theta}}\quad\mathrm{and}\quad Y_{\bar{\tau}_{\theta}}=L^{\xi}_{\bar{\tau}_{\theta}}.
\end{equation}
\end{enumerate}
\end{proposition}
\begin{proof}
Ad 1). Assume  that   $R^{-,*}$ is continuous. By the definition of  $\bar{\sigma}_{\theta}$ we have that  $R^{-}_{\bar{\sigma}_{\theta}}=R^{-}_{\theta}$. Thus, for any  $a\ge0$,
\[
Y_t=Y_{\bar{\sigma}_{\theta}}+\int^{\bar{\sigma}_{\theta}}_t f(r,Y_r,Z_r)\,dr+R^+_{\bar{\sigma}_{\theta}}-R^+_t-\int^{\bar{\sigma}_{\theta}}_t Z_r\,dB_r,\quad t\in[\theta,\bar{\sigma}_{\theta}].
\]
By Proposition  \ref{nonlinprop},   $Y$ is an  $\mathbb{E}^f$-supermartingale on $[[\theta,\bar{\sigma}_{\theta}]]$.
We shall prove  that $Y_{\bar{\sigma}_{\theta}}=U^{\xi}_{\bar{\sigma}_{\theta}}$. Assume that  $\bar{\sigma}_{\theta}<T$ (in case  $\bar{\sigma}_{\theta}=T$ the desired equality is obvious). Suppose, by contradiction, that   $P(Y_{\bar{\sigma}_{\theta}}<U^{\xi}_{\bar{\sigma}_{\theta}})>0$. 
By the {\em minimality condition},  $\Delta^+R^-_{\bar{\sigma}_{\theta}}=0$ on  $\{Y_{\bar{\sigma}_{\theta}}<U^{\xi}_{\bar{\sigma}_{\theta}}\}$. 
Observe that  $\Delta^+Y_{\bar{\sigma}_{\theta}}=-(\Delta^+R^+_{\bar{\sigma}_{\theta}}-\Delta^+R^-_{\bar{\sigma}_{\theta}})=-\Delta^+R^+_{\bar{\sigma}_{\theta}}\le 0$, which  implies that $Y$ is right upper semicontinuous on $\{Y_{\bar{\sigma}_{\theta}}<U^{\xi}_{\bar{\sigma}_{\theta}}\}$.
Let  $a\in\mathbb R$ and $\varepsilon>0$ (depending on  $\omega\in\Omega$) be such that  $U^{\xi}_{\bar{\sigma}_{\theta}}>a+\varepsilon$ 
and $Y_{\bar{\sigma}_{\theta}}<a-\varepsilon$. Since $Y$ is right upper semicontinuous on $\{Y_{\bar{\sigma}_{\theta}}<U^{\xi}_{\bar{\sigma}_{\theta}}\}$, and $U$ is right lower semicontinuous, there exists  $\delta>0$ (depending on  $\omega\in\Omega$) such that  $U^{\xi}_{\bar{\sigma}_{\theta}+s}>a+\varepsilon$ and $Y_{\bar{\sigma}_{\theta}+s}<a-\varepsilon,\, s\in [0,\delta]$. 
Furthermore, from the definition of $\bar{\sigma}_{\theta}$ we have  $R^{-,*}_{\bar{\sigma}_{\theta}+\delta}>R^{-,*}_{\bar{\sigma}_{\theta}}$. 
Consequently, on the set $\{Y_{\bar{\sigma}_{\theta}}<U^{\xi}_{\bar{\sigma}_{\theta}}\}$ the following holds
\[
\int^{\bar{\sigma}_{\theta}+\delta}_{\bar{\sigma}_{\theta}}(\underrightarrow{U}^{\xi}_r-Y_{r-})\,dR^{-,*}_r>2\varepsilon(R^{-,*}_{\bar{\sigma}_{\theta}+\delta}-R^{-,*}_{\bar{\sigma}_{\theta}})>0.
\]
This contradicts the {\em minimality condition}. 

What is left is to show that  $Y_{\sigma^*_{\theta}}=U^{\xi}_{\sigma^*_{\theta}}$. 
In case  $\sigma^*_{\theta}=T$ the  equation follows at once. Suppose that  $\sigma^*_{\theta}<T$.
If  $\Delta^+R^-_{\sigma^*_{\theta}}(\omega)>0$, then  by the very definition of a solution to RBSDE, we have 
$Y_{\sigma^*_{\theta}}(\omega)=U^{\xi}_{\sigma^*_{\theta}}(\omega)$. 
Suppose that  $\Delta^+R^-_{\sigma^*_{\theta}}(\omega)=0$. 
Observe that 
\[
\Delta^+Y_{\sigma^*_{\theta}}(\omega)=-(\Delta^+_{\sigma^*_{\theta}}R^+_{\sigma^*_{\theta}}(\omega)-\Delta^+R^-_{\sigma^*_{\theta}}(\omega))=-\Delta^+_{\sigma^*_{\theta}}R^+_{\sigma^*_{\theta}}(\omega)\le 0. 
\]
Thus, $Y_{\sigma^*_{\theta}+}(\omega)\le Y_{\sigma^*_{\theta}}(\omega)$. 
By the definition of  $\sigma^*_{\theta}$ there exists a non-increasing sequence  $(t_n(\omega))_{n\ge 1}$ 
such that  $t_n(\omega)\searrow\sigma^*_{\theta}(\omega)$ and $Y_{t_n}(\omega)=U^{\xi}_{t_n}(\omega)$. 
Letting $n\to \infty$ and using right lower semicontinuity of  $U$ we find that $Y_{\sigma^*_{\theta}+}(\omega)\ge U^{\xi}_{\sigma^*_{\theta}}(\omega)$, which combined  with  $Y_{\sigma^*_{\theta}+}(\omega)\le Y_{\sigma^*_{\theta}}(\omega)$ gives the result.

Ad A2). The case when  $R^{+,*}$ is supposed to be continuous runs analogously. 
\end{proof}

\begin{corollary}\label{luty06}
Under assumptions of  Proposition \ref{dyn23pr} we have that continuity of  $R^{-,*}$ (resp. $R^{+,*}$)  implies $\sigma^*_{\theta}\le\bar{\sigma}_{\theta}$ (resp. $\tau^*_{\theta}\le\bar{\tau}_{\theta}$).
\end{corollary}

\begin{proposition}\label{dyn21}
Let  $(Y,Z,R)$ be a solution to \textnormal{RBSDE}$^T(\xi,f,L,U)$. If  $L$ (resp. $U$) is left upper semicontinuous (resp. left lower semicontinuous), then   $R^{+,*}$ (resp. $R^{-,*}$) is continuous. 
\end{proposition}
\begin{proof}
Let  $\tau\in\mathcal{T}$ be predictable.  We shall prove that  $\Delta^-R^{+,*}_{\tau}=0$. 
We have 
\begin{equation}\label{dyn19}
\begin{split}
\Delta Y_{\tau}=-\Delta R^{+,*}_{\tau}+\Delta R^{-,*}_{\tau}=-\Delta R^{+,*}_{\tau}\mathbf{1}_{\{Y_{\tau-}=\overrightarrow{L}_{\tau}\}\cap D}
+\Delta R^{-,*}_{\tau}\mathbf{1}_{\{Y_{\tau-}=\underrightarrow{U}_{\tau}\}\cap D'},
\end{split}
\end{equation}
where  $D:=\{\Delta R^{+,*}_{\tau}>0\}$ and  $D':=\{\Delta R^{-,*}_{\tau}>0\}$. 
Since $dR^+\perp dR^-$,  $D\cap D'=\emptyset$. Thus, on the set  $D$, $\Delta Y_\tau\le 0.$
From this and the regularity assumption on $L$, $\overrightarrow{L}_{\tau}\le L_{\tau}\le Y_{\tau}\le Y_{\tau-}$ on $D$.
Consequently, $\Delta Y_\tau= 0$ on  $\{Y_{\tau-}=\overrightarrow{L}_{\tau}\}\cap D$. This combined with  \eqref{dyn19} implies $\Delta^-R^{+,*}_{\tau}=0$. Since the last inequality holds for any predictable $\tau\in\mathcal{T}$, we deduce that  $R^{+,*}$ is continuous. 
The similar reasoning may be applied to  $U$.
\end{proof}

\begin{theorem}\label{dyn26}
Suppose that  $L$ is upper semicontinuous and $U$ is lower semicontinuous. 
Let  $(Y,Z,R)$ be a solution to \textnormal{RBSDE}$^T(\xi,f,L,U)$.
Then for any  $\theta\in\mathcal{T}$ couples  \eqref{dyn14} and  \eqref{dyn15} are saddle points at  $\theta$ for the nonlinear Dynkin game
with the payoff  function \eqref{dyn271}.
\end{theorem}
\begin{proof}
Let  $\theta\in\mathcal{T}$. By Theorem  \ref{dyn20} $Y_{\theta}=\overline{V}_0(\theta)=\underline{V}_0(\theta)$. 
By Proposition  \ref{dyn21}   $R^{+,*}$, $R^{-,*}$ are continuous. Let  $\tau\in\mathcal T_\theta$.
Since  $\sigma^*_{\theta}\le\bar{\sigma}_{\theta}$ (see Corollary  \ref{luty06}), by Proposition \ref{dyn23pr} process $Y$ is an  $\mathbb{E}^f$-supermartingale on  $[[\theta,\tau\wedge\sigma^*_{\theta}]]$. Therefore,
\begin{equation}\label{dyn24}
Y_{\theta}\ge\mathbb{E}^f_{\theta,\tau\wedge\sigma^*_{\theta}}[Y_{\tau\wedge\sigma^*_{\theta}}].
\end{equation}
 Since  $Y\ge L$ and  $Y_{\sigma^*_{\theta}}=U_{\sigma^*_{\theta}}$ (see  Proposition  \ref{dyn23pr}), we also have 
 \[
 Y_{\tau\wedge\sigma^*_{\theta}}=Y_{\tau}\mathbf{1}_{\{\tau\le\sigma^*_{\theta}\}}+Y_{\sigma^*_{\theta}}\mathbf{1}_{\{\sigma^*_{\theta}<\tau\}}\ge L_{\tau}\mathbf{1}_{\{\tau\le\sigma^*_{\theta}\}}+U_{\sigma^*_{\theta}}\mathbf{1}_{\{\sigma^*_{\theta}<\tau\}}=J_0(\tau,\sigma^*_{\theta}).
 \]
 Using  \eqref{dyn24} and the fact that $\mathbb{E}^f$ is a non-decreasing operator, we deduce that   $Y_{\theta}\ge\mathbb{E}^f_{\theta,\tau\wedge\sigma^*_{\theta}}J_0(\tau,\sigma^*_{\theta})$ for any $\tau\in\mathcal T_{\theta}$, in particular  $\mathbb E^f_{\theta,\tau^*_{\theta}\wedge\sigma^*_{\theta}}J_0(\tau^*_{\theta},\sigma^*_{\theta})\le Y_{\theta}$. In the similar  way we arrive at  $Y_{\theta}\le\mathbb{E}^f_{\theta,\tau^*_{\theta}\wedge\sigma}J_0(\tau^*_\theta,\sigma)$ for any $\sigma\in\mathcal{T}_{\theta}$, and so $Y_{\theta}\le \mathbb E^f_{\theta,\tau^*_{\theta}\wedge\sigma^*_{\theta}}J_0(\tau^*_{\theta},\sigma^*_{\theta})$. Consequently,  $Y_{\theta}=\mathbb E^f_{\theta,\tau^*_{\theta}\wedge\sigma^*_{\theta}}J(\tau^*_{\theta},\sigma^*_{\theta})$ and  $(\tau^*_{\theta},\sigma^*_{\theta})$ is a saddle point at  $\theta$. Analogously, one  shows,  by using Proposition  \ref{dyn23pr}, that  $(\bar{\tau}_{\theta},\bar{\sigma}_{\theta})$ is a saddle point at  $\theta$.
\end{proof}

\section{Existence result}\label{roz9}

In the whole section, we assume   that $L,U$ are $\mathbb{F}$-optional processes of class (D).

Let  us consider the following assumption, which is called in the literature  {\em weak Mokobodzki's condition}.

\begin{enumerate}
\item[(WM)] There exists a  semimartingale $X$ such that $L\le X\le U$.
\end{enumerate}

\begin{proposition}
Assume that $L, U$ are left-limited, and 
\begin{equation}
\label{eq.lus}
 \overleftarrow{L}_t<\underleftarrow{U}_t,\quad L_{t-}<U_{t-},\quad t\in [0,T].
\end{equation}
Then weak Mokobodzki's condition \textnormal{(WM)} holds for $L,U$.
\end{proposition}
\begin{proof}
We let  $\tau_0:=0$, and for $n\ge 1$,
\[
\tau_n:=\inf\{t>\tau_{n-1}:  (\overleftarrow{L}_{\tau_{n-1}}+\underleftarrow{U}_{\tau_{n-1}})<2L_t\text{ or }  (\overleftarrow{L}_{\tau_{n-1}}+\underleftarrow{U}_{\tau_{n-1}})>2U_t\}\wedge T.
\]
Obviously, $(\tau_n)$ is nondecreasing. Observe that by the definition of $\tau_n$ for each $\omega\in\Omega$
there exists a sequence $\{t^n_m\}$ such that $t^n_m\searrow\tau_n(\omega)$ and for all $m\in\mathbb{N}$
\begin{equation}
\label{eq.lu1}
\begin{split}
&\overleftarrow{L}_{\tau_{n-1}(\omega)}(\omega)+\underleftarrow{U}_{\tau_{n-1}(\omega)}(\omega)<2L_{t^n_m}(\omega)\quad\text{or}\\
&\overleftarrow{L}_{\tau_{n-1}(\omega)}(\omega)+\underleftarrow{U}_{\tau_{n-1}(\omega)}(\omega)>2U_{t^n_m}(\omega).
\end{split}
\end{equation}
Letting $m\to \infty$ yields 
\begin{equation}
\label{eq.lu2}
\begin{split}
&\overleftarrow{L}_{\tau_{n-1}(\omega)}(\omega)+\underleftarrow{U}_{\tau_{n-1}(\omega)}(\omega)\le 2\overleftarrow{L}_{\tau_{n}(\omega)}(\omega)\quad\text{or}\\
&\overleftarrow{L}_{\tau_{n-1}(\omega)}(\omega)+\underleftarrow{U}_{\tau_{n-1}(\omega)}(\omega)\ge 2\underleftarrow{U}_{\tau_{n}(\omega)}(\omega),\quad n\ge 1.
\end{split}
\end{equation}

{\bf Step 1}. We shall prove that $(\tau_n)$ is a chain.
First, note that
\begin{equation}
\label{eq.ul3}
P(\tau_{n-1}=\tau_n<T)=0,\quad n\ge 1.
\end{equation}
Indeed, suppose that for some $n\ge 1$, $P(\tau_{n-1}=\tau_n<T)>0$.
Let  $\omega\in\{\tau_{n-1}=\tau_n<T\}$. Then, 
by \eqref{eq.lu2}
\[
\underleftarrow{U}_{\tau_{n}(\omega)}(\omega)\le \overleftarrow{L}_{\tau_{n}(\omega)}(\omega).
\]
Therefore, $P(\underleftarrow{U}_{\tau_{n}}\le \overleftarrow{L}_{\tau_{n}})>0$,
which contradicts   \eqref{eq.lus}.  Suppose that  $(\tau_n)$ is not a chain. Then, according to \eqref{eq.ul3},
there must exist $\tau\in\mathcal{T}$ such that $\tau_n\nearrow\tau$ and $P(\bigcap^{\infty}_{n=1}\{\tau_n<\tau\})>0$. 
Let $\omega\in\bigcap^{\infty}_{n=1}\{\tau_n<\tau\}$. 
By the second inequality in \eqref{eq.lus} for any $\delta>0$ there exists $n_\delta\ge 1$  such that 
\[
\overleftarrow{L}_{\tau_{n-1}(\omega)}(\omega)\le L_{\tau(\omega)-}(\omega)+\delta, \quad
U_{\tau(\omega)-}(\omega)-\delta\le \underleftarrow{U}_{\tau_{n-1}(\omega)}(\omega),\quad n\ge n_\delta.
\]
Suppose that the first inequality in \eqref{eq.lu2} holds for infinitely many $n\ge 1$ (the proof in the second case is analogous).
Then, by the above inequalities, we conclude from  \eqref{eq.lu2} that 
\[
L_{\tau(\omega)-}(\omega)+U_{\tau(\omega)-}(\omega)-2\delta\le L_{\tau(\omega)-}(\omega)+\delta.
\]
Letting $\delta\searrow 0$, we  obtain that $U_{\tau(\omega)-}(\omega)\le  L_{\tau(\omega)-}(\omega)$.
Therefore, $P(U_{\tau-}\le  L_{\tau-})>0$,
which contradicts  \eqref{eq.lus}. Thus, $(\tau_n)$ is a chain.

{\bf Step 2.} We shall construct a semimartingale lying between barriers $L,U$.
Define
\[
X_t:= \frac12\sum_{n=1}^\infty\Big((\overleftarrow{L}_{\tau_{n-1}}+\underleftarrow{U}_{\tau_{n-1}})\mathbf1_{(\tau_{n-1},\tau_n)}(t)+(L_{\tau_{n-1}}+U_{\tau_{n-1}})\mathbf1_{\{\tau_{n-1}\}}(t)\big),
\]
for $t\in [0,T]$. Clearly, $L_t\le X_t\le U_t,\, t\in [0,T]$ and $X$ is $\mathbb F$-adapted. Since $(\tau_n)$
is a chain, we get that $X$ is of  finite variation, thus a semimartingale. This completes the proof.
\end{proof}

\begin{proposition}\label{19listopada1}
Assume that   $f_1,f_2$   satisfy \textnormal{(H1)--(H4)}, $\xi_1,\xi_2,f_1,f_2$ satisfy \textnormal{(H5)} with $p=2$, 
and $|f^1-f^2|(\cdot,Y^2,Z^2)\in L^{1,2}_{\mathbb{F}}(0,T)$.
 Let $(Y^i,Z^i,R^i)$ be a solution to \textnormal{RBSDE}$^T(\xi_i,f^i,L,U)$ such that $Y^i\in \mathcal{S}^2_{\mathbb{F}}(0,T)$, $i=1,2$. 
 Then there exists $c>0$, depending only on $T,\mu,\lambda$,  such that
\begin{equation}
\label{eq.stabr2}
\begin{split}
\|Y^1-Y^2\|_{\mathcal D^2(0,T)}\le c \big( \|\xi_1-\xi_2\|_{L^2}+\||f_1-f_2|(\cdot,Y^2,Z^2)\|_{L^{1,2}_{\mathbb{F}}(0,T)}\big).
\end{split}
\end{equation}
\end{proposition}
\begin{proof}
We let $J^i$ denote the right-hand side of \eqref{dyn27} but with $\xi$
replaced by $\xi_i,\, i=1,2$. Let
\[
\tilde f(t,y,z):= f_1(t,Y^1_t,Z^1_t)-f_2(t,Y^1_t,Z^1_t)+f_2(t,y,z),\quad t\in [0,T],\,y\in\mathbb R,\, z\in \mathbb R^d.
\]
Observe that $(Y^1,Z^1,R^1)$ is a solution to RBSDE$^T(\xi_1,\tilde f,L,U)$. By Theorem \ref{17sierpnia1}
\[
Y^1_\theta=\mathop{\mathrm{ess\,sup}}_{\tau\lfloor A\in\mathcal{U}_{\theta}}\mathop{\mathrm{ess\,inf}}_{\sigma\lfloor B\in\mathcal{U}_{\theta}}\mathbb{E}^{\tilde f}_{\theta,\tau\wedge\sigma}J^1(\tau\lfloor A,\sigma\lfloor B)
\]
and 
\[
Y^2_\theta=\mathop{\mathrm{ess\,sup}}_{\tau\lfloor A\in\mathcal{U}_{\theta}}\mathop{\mathrm{ess\,inf}}_{\sigma\lfloor B\in\mathcal{U}_{\theta}}\mathbb{E}^{f_2}_{\theta,\tau\wedge\sigma}J^2(\tau\lfloor A,\sigma\lfloor B).
\]
Hence
\begin{equation}
\label{eq.mm11}
|Y^1_\theta-Y^2_\theta|\le \mathop{\mathrm{ess\,sup}}_{\tau\lfloor A\in\mathcal{U}_{\theta}}\mathop{\mathrm{ess\,sup}}_{\sigma\lfloor B\in\mathcal{U}_{\theta}}
\Big|\mathbb{E}^{\tilde f}_{\theta,\tau\wedge\sigma}J^1(\tau\lfloor A,\sigma\lfloor B)-\mathbb{E}^{f_2}_{\theta,\tau\wedge\sigma}J^2(\tau\lfloor A,\sigma\lfloor B)\Big|.
\end{equation}
Applying Proposition \ref{nonlinprop}(vi) yields
\begin{align*}
\mathbb E\Big|\mathbb{E}^{\tilde f}_{\theta,\tau\wedge\sigma}J^1(\tau\lfloor A,\sigma\lfloor B)&-\mathbb{E}^{f_2}_{\theta,\tau\wedge\sigma}J^2(\tau\lfloor A,\sigma\lfloor B)\Big|^2
\\&\le  c\mathbb E\Big[|\xi_1-\xi_2|^2+\Big(\int^T_0|f_1-f_2|(r,Y^2_r,Z^2_r)\,dr\Big)^2\Big].
\end{align*}
Combining the last two inequalities gives at once the result. 
\end{proof}

\begin{theorem}
Assume that \textnormal{(H1)-(H4),(H7)} are in force. Suppose that  \textnormal{(H5)} is satisfied with  $p=1$.
Then there exists a  solution $(Y,Z,R)$ to \textnormal{RBSDE}$^T(\xi,f,L,U)$.
\end{theorem}
\begin{proof}
Let $X$ be the process appearing in (H7). Since $X$ is a special semimartingale, there exists a chain $(\hat\gamma_k)$
and processes   $H\in\mathcal{H}_{\mathbb{F}}(0,T)$ and $C\in\mathcal{V}_{\mathbb{F}}(0,T)$ such that $H\in\mathcal{H}^2_{\mathbb{F}}(0,\hat\gamma_k)$, $C\in\mathcal{V}^2_{\mathbb{F}}(0,\hat\gamma_k),\, k\ge 1$, and 
\[
X_t=X_0+C_t+\int_0^t H_r\,dB_r,\quad t\in[0,T].
\]
Let
\[
f_{n,m}(t,y,z)=\{f(t,y,z)\wedge n\}\vee (-m).
\]
Note that $f_{n,m}$ is nondecreasing with respect to $n$ and nonincreasing with respect to $m$. Moreover, $f_{n,m}(t,y,z)\nearrow f_{m}(t,y,z):=f(t,y,z)\vee (-m)$, when $n\to\infty$, and $f_m(t,y,z)\searrow f(t,y,z)$, when $m\to\infty$. By \cite[page 417]{DM1} there exist regulated processes  $\hat{L}$, $\hat{U}$  satisfying
\[
\hat{L}_{\alpha}=\mathop{\mathrm{ess\,inf}}_{\tau\in\mathcal{T}_{\alpha}}\mathbb{E}(L_{\tau}|\mathcal{F}_{\alpha}),\quad\hat{U}_{\alpha}=\mathop{\mathrm{ess\,sup}}_{\tau\in\mathcal{T}_{\alpha}}\mathbb{E}(U_{\tau}|\mathcal{F}_{\alpha}),\,\alpha\in\mathcal{T}.
\]
Furthermore, by \cite[Proposition 3.8]{KRzS}, $-\hat{L}$, $\hat{U}$ are supermartingales of class (D) on $[0,T]$. 
As a result, there exist processes $F,G\in\mathcal{H}_{\mathbb{F}}(0,T)$ and $A,D\in\mathcal{V}^{+,1}_{\mathbb{F}}(0,T)$ such that
\[
\hat{L}_t=\hat{L}_T-\int^T_t\,dA_r-\int^T_t F_r\,dB_r,\quad t\in[0,T].
\]
and
\[
\hat{U}_t=\hat{U}_T-\int^T_t\,dD_r-\int^T_t G_r\,dB_r,\quad t\in[0,T].
\]
Obviously, $\hat{L}\le L\le U\le\hat{U}$. Since $\hat{L}$ and $\hat{U}$ are of class (D), by (H4) there exists a chain $(\tau^1_k)$ on $[0,T]$ such that
\begin{equation}\label{7pazdziernika1}
\mathbb{E}\Big(\int^{\tau_k}_0|f(r,\hat{L}_r,0)|\,dr\Big)^2+\mathbb{E}\Big(\int^{\tau_k}_0|f(r,\hat{U}_r,0)|\,dr\Big)^2\le k.
\end{equation}
Moreover, let us consider  chain $(\tau^2_k)$ on $[0,T]$ such that $\hat{L},\hat{U}\in\mathcal{S}^2_{\mathbb{F}}(0,\tau^2_k)$,  $f(\cdot,0,0)\in L^{1,2}_{\mathbb{F}}(0,\tau_k^2)$, and   $A,D\in\mathcal{V}^2_{\mathbb{F}}(0,\tau^2_k),\, k\ge 1$.
We let  $\gamma_k:=\hat \gamma_k\wedge\tau^1_k\wedge\tau^2_k$ 
Define
\[
L^n_t=L_t\mathbf{1}_{\{t\le\gamma_n\}}+\hat{L}\mathbf{1}_{\{t>\gamma_n\}},\quad U^n_t=U_t\mathbf{1}_{\{t\le\gamma_n\}}+\hat{U}_t\mathbf{1}_{\{t>\gamma_n\}}.
\]
Note that
\begin{equation}\label{7pazdziernika2}
\hat{L}\le L^n\le L^{n+1}\le L\le U\le U^{n+1}\le U^n\le\hat{U},\quad n\ge 1.
\end{equation}
Moreover, $L^n\nearrow L$ and $U^n\searrow U$. Finally, we define
\[
X^{n,m}_t=X_t\mathbf{1}_{\{t\le\gamma_n\wedge\gamma_m\}}+\hat{L}_t\mathbf{1}_{\{t>\gamma_n\ge\gamma_m\}}+\hat{U}_t\mathbf{1}_{\{t>\gamma_m>\gamma_n\}}.
\]
Note that $L^n\le X^{n,m}\le U^m$ and $X^{n,m}$ is a difference of two supermartingales of class (D).
Therefore, by the definition of $f_{n,m}$, strong Mokobodzki's condition \textnormal{(H6*)} is satisfied with $L^n$, $U^m$, $X^{n,m}$ and $f_{n,m}$. By Theorem \ref{28wrzesnia1} there exists a unique solution $(Y^{n,m},Z^{n,m},R^{n,m})$ to \textnormal{RBSDE}$^T(\xi,f_{n,m},L^n,U^m)$ such that $Y^{n,m}$ is of class \textnormal{(D)}, $Z^{n,m}\in\mathcal{H}^q_{\mathbb{F}}(0,T)$, $q\in(0,1)$ and $R^{n,m}\in\mathcal{V}^1_{0,\mathbb{F}}(0,T)$. By \cite[Proposition 3.2, Lemma 3.3]{KRzS2} , $Y^{n,m}$ is nondecreasing with respect to $n$ and nonincreasing with respect to $m$. Let us define
\[
Y^m=\sup_{n\ge 1}Y^{n,m},\quad Y=\inf_{m\ge 1} Y^m.
\]
Obviously, $Y^m$ and $Y$ are of class (D). The remainder of the proof, we divide into two steps.

\textbf{Step 1.} We shall prove that or any $k\le m$, process $Y^m$ is the first component of a solution to \textnormal{RBSDE}$^{\gamma_k}(Y^m_{\gamma_k},f_m,L,U^m)$. Let $k\le  m\le n$. 
Since $\hat{L}\le Y^{n,m}\le \hat{U}$, we have  $Y^{n,m},Y^m\in\mathcal{S}^2_{\mathbb{F}}(0,\gamma_k)$.  
According to  Theorem \ref{18listopada1} - observe that  (H5), (H6) are satisfied with $L,U^m,f_m,Y^m,X$ and $p=2$ on $[[0,\gamma_k]]$ 
- there exists a solution  $(\tilde{Y}^{k,m},\tilde{Z}^{k,m},\tilde{R}^{k,m})$ to 
RBSDE$^{\gamma_k}(Y^m_{\gamma_k},f_m,L,U^m)$ such that $\tilde{Y}^{k,m}\in\mathcal{S}^2_{\mathbb{F}}(0,\gamma_k)$, 
$\tilde{Z}^{k,m}\in\mathcal{H}^2_{\mathbb{F}}(0,\gamma_k)$ and $\tilde{R}^{k,m}\in\mathcal{V}^2_{0,\mathbb{F}}(0,\gamma_k)$. 
We shall show that $Y^m=\tilde{Y}^{k,m}$ on $[[0,\gamma_k]]$. By Proposition \ref{19listopada1}
\begin{equation}\label{4grudnia1}
\begin{split}
&\|\tilde{Y}^{k,m}-Y^{n,m}\|^2_{\mathcal D^2(0,\gamma_k)}\le c\mathbb{E} \Big[\Big(\int^{\gamma_k}_0|f_m-f_{n,m}|(r,\tilde{Y}^{k,m}_r,\tilde{Z}^{k,m}_r)\,dr\Big)^2\Big]\\
&\quad+|\tilde{Y}^{k,m}_{\gamma_k}-Y^{n,m}_{\gamma_k}|^2= c\mathbb{E}\Big[|Y^m_{\gamma_k}-Y^{n,m}_{\gamma_k}|^2\\
&\quad+\Big(\int^{\gamma_k}_0|f(r,\tilde{Y}^{k,m}_r,\tilde{Z}^{k,m}_r)|\mathbf{1}_{\{f(r,\tilde{Y}^{k,m}_r,\tilde{Z}^{k,m}_r)>n\}}\,dr\Big)^2\Big].
\end{split}
\end{equation}
Observe that $0\le Y^m_{\gamma_k}-Y^{n,m}_{\gamma_k}\le Y^m_{\gamma_k}\in L^2(\Omega,\mathcal{F}_{\gamma_k})$
(the last assertion is a consequence of the fact that $Y^m\in\mathcal{S}^2_{\mathbb{F}}(0,\gamma_k)$).
Therefore, by the Lebesgue dominated  convergence theorem, the first term on the right-hand side of  \eqref{4grudnia1}
tends to zero as $n\to \infty$.
Note that, by the definition of $\gamma_k$, (H1), \eqref{7pazdziernika1} and Jensen's inequality
\begin{equation*}
\begin{split}
&\mathbb{E}\Big(\int^{\gamma_k}_0|f(r,\tilde{Y}^{k,m}_r,\tilde{Z}^{k,m}_r)|\,dr\Big)^2
\le \mathbb{E}\Big(\lambda\int^{\gamma_k}_0|\tilde{Z}^{k,m}_r|\,dr\Big)^2\\
&\quad\quad+\mathbb{E}\Big(\int^{\gamma_k}_0|f(r,\tilde{Y}^{k,m}_r,0)|\,dr\Big)^2\le T\lambda^2\mathbb{E}\int^{\gamma_k}_0|\tilde{Z}^{k,m}_r|^2\,dr
\\
&\quad\quad+\mathbb{E}\Big(\int^{\gamma_k}_0|f(r,\hat{L}_r,0)|+|f(r,\hat{U}_r,0)|\,dr\Big)^2\\
&\quad\le T\lambda^2\mathbb{E}\int^{\gamma_k}_0|\tilde{Z}^{k,m}_r|^2\,dr+T^2\cdot k^2<\infty.
\end{split}
\end{equation*}
Consequently, by the   Lebesgue dominated convergence theorem, the most right term in \eqref{4grudnia1}
tends to zero as $n\to \infty$. As a result, letting  $n\to\infty$ in \eqref{4grudnia1}, 
we obtain that $Y^{n,m}\to \tilde{Y}^{k,m}$ in $\mathcal D^2_{\mathbb{F}}(0,\gamma_k)$. 
This completes the proof of step 1.

\textbf{Step 2.} We shall prove that $Y$ is the first component of a solution to  \textnormal{RBSDE}$(\xi,f,L,U)$. Let $k\le m$. 
Since $\hat{L}\le Y^{m}\le \hat{U}$, we have that $Y^{m},Y\in\mathcal{S}^2_{\mathbb{F}}(0,\gamma_k)$. 
Observe that conditions  (H5) and (H6) are met by  $L,U,f, X$ on $[[0,\gamma_k]]$ with $p=2$.
Therefore, by   Theorem \ref{18listopada1},
there exists a solution  $(\tilde{Y}^k,\tilde{Z}^k,\tilde{R}^k)$  to RBSDE$^{\gamma_k}(Y_{\gamma_k},f,L,U)$ such that 
$\tilde{Y}^k\in\mathcal{S}^2_{\mathbb{F}}(0,\gamma_k)$, 
$\tilde{Z}^k\in\mathcal{H}^2_{\mathbb{F}}(0,\gamma_k)$, 
and $\tilde{R}^k\in\mathcal{V}^2_{0,\mathbb{F}}(0,\gamma_k)$. 
We shall show that $Y=\tilde{Y}^k$ on $[[0,\gamma_k]]$. By Proposition \ref{19listopada1}
\begin{equation}\label{4grudnia2}
\begin{split}
\|\tilde{Y}^k-Y^m\|_{\mathcal D^2(0,\gamma_k)}&\le C\mathbb{E}\Big[\Big(\int^{\gamma_k}_0|f(r,\tilde{Y}^k_r,\tilde{Z}^k_r)-f_m(r,\tilde{Y}^k_r,\tilde{Z}^k_r)|\,dr\Big)^2
\\+|Y_{\gamma_k}-Y^m_{\gamma_k}|^2\Big]&= C\mathbb{E}\Big[\Big(\int^{\gamma_k}_0|f(r,\tilde{Y}^k_r,\tilde{Z}^k_r)|
\mathbf{1}_{\{f(r,\tilde Y^k_r,\tilde Z^k_r)<-m\}}\,dr\Big)^2\\
&\quad+|Y_{\gamma_k}-Y^m_{\gamma_k}|^2\Big].
\end{split}
\end{equation}
By the Lebesgue dominated convergence theorem the first term on the right-hand side of \eqref{4grudnia2} tends to zero
when $m\to \infty$. By combining \eqref{7pazdziernika1}, the definition of $\gamma_k$,  condition (H1), and Jensen's inequality,
we conclude that
\begin{equation*}
\begin{split}
\mathbb{E}\Big(\int^{\gamma_k}_0&|f(r,\tilde{Y}^k_r,\tilde{Z}^k_r)|\,dr\Big)^2
\le \mathbb{E}\Big(\lambda\int^{\gamma_k}_0|\tilde{Z}^k_r|\,dr\Big)^2+\mathbb{E}\Big(\int^{\gamma_k}_0|f(r,\tilde{Y}^k_r,0)|\,dr\Big)^2\\
&\le T\lambda^2\mathbb{E}\int^{\gamma_k}_0|\tilde{Z}^k_r|^2\,dr
+\mathbb{E}\Big(\int^{\gamma_k}_0|f(r,\hat{L}_r,0)|+|f(r,\hat{U}_r,0)|\,dr\Big)^2\\
&\le T\lambda^2\mathbb{E}\int^{\gamma_k}_0|\tilde{Z}^k_r|^2\,dr+T^2\cdot k^2<\infty.
\end{split}
\end{equation*}
Therefore, by the Lebesgue dominated convergence theorem, the most right term in \eqref{4grudnia2}
tends to zero as $m\to \infty$. Consequently, letting $m\to \infty$ in \eqref{4grudnia2}, we deduce that 
 $Y^m\to \tilde{Y}^k$ in $\mathcal D^2_{\mathbb{F}}(0,\gamma_k)$. 
Hence, $\tilde{Y}^k=Y$ on $[0,\gamma_k],\, k\ge 1$. 
In other words, for any $k\ge 1$, process $Y$ is the first component of a solution to RBSDE$^{\gamma_k}(Y_{\gamma_k},f,L,U)$.
This in turn implies, by using the uniqueness argument (see Theorem \ref{18listopada1}), that $\tilde Z^k=\tilde Z^{k+1}$, and $\tilde R^k=\tilde R^{k+1}$ on $[0,\gamma_k],\, k\ge 1$. 
With the aid of these properties, one easily checks that the triple $(Y,M,R)$ is a solution to RBSDE$^{\gamma_k}(Y_{\gamma_k},f,L,U)$ for each $k\ge 1$, where
\[
Z_t:=\sum_{k=0}^\infty \tilde Z^k_t\mathbf1_{(\gamma_k,\gamma_{k+1}]}(t),\quad R_t:=\sum_{k=0}^\infty \tilde R^k_t\mathbf1_{(\gamma_k,\gamma_{k+1}]}(t).
\]
This combined with the fact that $(\gamma_k)$ is a chain implies that $(Y,Z,R)$ is a solution to RBSDE$^T(\xi,f,L,U)$.

\end{proof}

\subsection*{Acknowledgements}
{\small T. Klimsiak is supported by Polish National Science Centre: Grant No. 2017/25/B/ST1/00878. M. Rzymowski acknowledges the support of the Polish National Science Centre: Grant No. 2018/31/N/ST1/00417.}


\end{document}